\documentclass[final,leqno,onefignum,onetabnum]{siamltex1213}

\usepackage{tikz}
\usetikzlibrary{arrows,chains,matrix,positioning,scopes}
\makeatletter
\tikzset{join/.code=\tikzset{after node path={%
\ifx\tikzchainprevious\pgfutil@empty\else(\tikzchainprevious)%
edge[every join]#1(\tikzchaincurrent)\fi}}}
\makeatother
\tikzset{>=stealth',every on chain/.append style={join},
         every join/.style={->}}
\tikzstyle{labeled}=[execute at begin node=$\scriptstyle,
   execute at end node=$]

\usepackage[hyphenbreaks]{breakurl}
\usepackage{cite}
\usepackage{amsmath, amssymb, amsfonts,xcolor}
\usepackage{array}

\allowdisplaybreaks
\newtheorem{example}{\textit{Example}}[section]

\usepackage{rotating}
\usepackage{tablefootnote}

\usepackage{algorithm2e}
\usepackage[noend]{algpseudocode}

\usepackage[multiple]{footmisc}

\title{Lasserre hierarchy for large scale polynomial\\ optimization in real and complex variables}

\author{C\'edric Josz\footnotemark[2]\ \footnotemark[4]\
\and Daniel~K. Molzahn\footnotemark[3]
}

\usepackage{soul}

\begin{document}
\maketitle

\renewcommand{\thefootnote}{\fnsymbol{footnote}}

\footnotetext[2]{French National Research Institute in Scientific Computing INRIA, Paris-Rocquencourt, BP 105, F-78153 Le Chesnay, France.}
\footnotetext[3]{Department of Electrical Engineering and Computer Science, University of Michigan, Ann Arbor, MI 48109, USA (\email{molzahn@umich.edu}). Support from Dow Sustainability Fellowship, ARPA-E grant DE-AR0000232 and Los Alamos National Laboratory subcontract 270958.}
\footnotetext[4]{French Transmission System Operator RTE, 9, rue de la Porte de Buc, BP 561, F-78000 Versailles, France (\email{cedric.josz@rte-france.com}). The research was funded by the CIFRE ANRT 
contract 2013/0179 and by the European Research Council (ERC) under the European Union's Horizon 2020 research and innovation program (grant agreement 666981 TAMING).}

\renewcommand{\thefootnote}{\arabic{footnote}}

\slugger{mms}{xxxx}{xx}{x}{x--x}

\begin{abstract}
We propose general notions to deal with large scale polynomial optimization problems and demonstrate their efficiency on a key industrial problem of the twenty first century, namely the optimal power flow problem. These notions enable us to find global minimizers on instances with up to 4,500 variables and 14,500 constraints.
First, we generalize the Lasserre hierarchy from real to complex to numbers in order to enhance its tractability when dealing with complex polynomial optimization. Complex numbers are typically used to represent oscillatory phenomena, which are omnipresent in physical systems. 
Using the notion of hyponormality in operator theory, we provide a finite convergence criterion which generalizes the Curto-Fialkow conditions of the real Lasserre hierarchy. Second, we introduce the multi-ordered Lasserre hierarchy in order to exploit sparsity in polynomial optimization problems (in real or complex variables) while preserving global convergence. It is based on two ideas: 1) to use a different relaxation order for each constraint, and 2) to iteratively seek a closest measure to the truncated moment data until a measure matches the truncated data. Third and last, we exhibit a block diagonal structure of the Lasserre hierarchy in the presence of commonly encountered symmetries.


\end{abstract}

\begin{keywords}
Multi-ordered Lasserre hierarchy,
Hermitian sum-of-squares,
chordal sparsity,
semidefinite programming,
optimal power flow.
\end{keywords}

\begin{AMS}90C22, 90C06, 90C26, 28A99, 14Q99, 47N10.\end{AMS}

\pagestyle{myheadings}
\thispagestyle{plain}
\markboth{C\'EDRIC JOSZ AND DANIEL~K. MOLZAHN}{Large Scale Polynomial Optimization}
\section{Introduction}

Polynomial optimization encompasses NP-hard non-convex problems that arise in various applications and it includes, as special cases, integer programming and quadratically-constrained quadratic programming. The Lasserre hierarchy~\cite{lasserre-2001,parrilo-2000b,parrilo-2003}, which draws on algebraic geometry~\cite{putinar-1993}, enables one to solve such problems to global optimality using semidefinite programming. A big challenge today is to make it applicable to large scale real world problems. Recent approaches in this direction include the use of chordal sparsity \cite{waki-2006}, the BSOS hierarchy \cite{toh-2017} and Sparse-BSOS hierarchy \cite{weisser-2017}, the DSOS and SDSOS hierarchies \cite{ahmadi-2014,kuang-2017bis,josz2017}, and ADMM for sum-of-squares \cite{zheng2017}. The Lasserre hierarchy has two dual facets, moments and sums-of-squares, and most approaches to reduce the computational burden can be viewed as a restriction on the sum-of-squares: \cite{waki-2006} restricts the number of variables, \cite{toh-2017} restricts the degree, and \cite{ahmadi-2014} restricts the number of terms inside the square. Following this line of research, we propose to restrict sum-of-squares to \textit{Hermitian} sum-of-squares \cite{angelo-2008} for optimization problems with oscillatory phenomena (e.g. power systems\cite{carpentier-1962,lavaei-low-2012,bai-fujisawa-wang-wei-2008},  imaging science~\cite{singer-2011,candes-2013,bandeira-2014}, signal processing~\cite{maricic-2003,aittomaki-2009,chen-2009,luo-2010}, automatic control~\cite{toker-1998}, and quantum mechanics~\cite{hilling-2010}). In addition, we propose to restrain the use of high degree sum-of-squares to only some constraints by using a different degree for each constraint. Finally, we show that if the polynomials defining the objective and the constraints are even (i.e. all the monomials have an even degree), then we can restrict the sum-of-squares to be even at no loss of bound quality. We show that a similar result holds for Hermitian sum-of-squares. The relevance of the restrictions to sum-of-squares that we propose is demonstrated on the optimal power flow problem in electrical engineering.

The optimal power flow is a central problem in power systems introduced half a century ago in~\cite{carpentier-1962}. 
It seeks to find a steady state operating point of an alternating current transmission network that respects Kirchhoff's laws, Ohm's law, and power balance equations. In addition, the point has to be optimal under a criterion such as total power generation or generation costs. It must also satisfy operational constraints which include narrow voltage ranges around nominal values and line ratings to keep Joule heating to acceptable levels. 
While many non-linear methods~\cite{murillosanchez-thomas-zimmerman-2011,castillo-2013} have been developed to solve this difficult problem, there is a strong motivation for producing more reliable tools. 
First, power systems are growing in complexity due to the increase in the share of renewables, the increase in the peak load, and the expected wider use of demand response and storage. Second, new tools are needed to profit from high-performance computing and advanced telecommunications (phasor measurement units, dynamic line ratings, etc.). Finally, the ultimate goal is to solve large problems 
(e.g. 10,000 buses in the synchronous grid of Continental Europe) 
with combinatorial complexity due to phase-shifting transformers, high-voltage direct current, and special protection schemes. Solving the continuous case (i.e., optimal power flow) to global optimality would be of great benefit to that end.
Since 2006, semidefinite and second-order conic relaxations have been proposed~\cite{jabr2006, low_tutorial, andersen2014, coffrin2015, taylor2015, dan2015}. It has emerged that 
the only approach that systematically yields global minimizers is the Lasserre hierarchy \cite{cedric_tps,pscc2014,ibm_paper}, although so far only for medium sized problems \cite{mh_sparse_msdp}. We solve large scale instances within minutes thanks to the restrictions of sum-of-squares discussed above.

This paper is organized as follows. Section \ref{sec:Complex Lasserre hierarchy} generalizes the Lasserre hierarchy to complex numbers to deal with complex polynomial optimization. Asymptotic convergence is discussed in Section \ref{sec:Asymptotic convergence}, while finite convergence is studied in Sections \ref{sec:Finite convergence} and \ref{sec:Truncated moment problem}. Sparsity is exploited in real and complex numbers via the multi-ordered Lasserre hierarchy in Section \ref{sec:Multi-ordered Lasserre hierarchy}, and symmetry is exploited via the block diagonal Lasserre hierarchy in Section \ref{sec:Block diagonal Lasserre hierarchy}. Finally, Section \ref{sec:Conclusion} concludes our work.

%
%

\section{Complex Lasserre hierarchy} 
\label{sec:Complex Lasserre hierarchy}
Consider the problem of finding global solutions to a complex polynomial optimization problem
\begin{equation}
\begin{array}{ll}
\inf\limits_{z \in \mathbb{C}^n} ~ & f(z,\bar{z}) ~:= \sum\limits_{\alpha,\beta} f_{\alpha,\beta} z^\alpha \bar{z}^\beta \\[1em]
\mathrm{s.t.} & g_i(z,\bar{z}) := \sum\limits_{\alpha,\beta} g_{i,\alpha,\beta} z^\alpha \bar{z}^\beta \geqslant 0, \quad i=1,\ldots,m.
\end{array}
\end{equation}
We use the multi-index notation $z^\alpha := z_1^{\alpha_1} \cdots z_n^{\alpha_n}$ for $z \in {\mathbb C}^n$,
$\alpha \in {\mathbb N}^n$, and $\bar{z}$ stands for the conjugate of $z$. As usual, $\mathbb{C}$ denotes the set of complex numbers (with $\textbf{i}$ the imaginary number) and $\mathbb{R}$ denotes the set of real numbers. The functions $f, g_1, \ldots, g_m$ are real-valued polynomials so that in the above sums only a finite number of coefficients $f_{\alpha,\beta}$ and $g_{i,\alpha,\beta}$ are nonzero and they satisfy $\overline{f_{\alpha,\beta}} = f_{\beta,\alpha}$ and $\overline{g_{i,\alpha,\beta}} = g_{i,\beta,\alpha}$. The feasible set is defined as $K:=\{z \in {\mathbb C}^n \: :\: g_i(z,\bar{z}) \geqslant 0, \: i=1,\ldots,m\}$.

\begin{example}
\normalfont
\textit{The optimal power flow problem is a complex polynomial optimization problem. It reads:}
$$
\inf\limits_{ z \in \mathbb{C}^n } ~~~~~ \sum\limits_{i=1}^n ~~  C_{i2} \left( \sum\limits_{j=1}^n \frac{\overline{Y}_{ij}}{2} z_i \overline{z}_j + \frac{Y_{ij}}{2} z_j \overline{z}_i \right)^2 ~+~ C_{i1}\left( \sum\limits_{j=1}^n \frac{\overline{Y}_{ij}}{2} z_i \overline{z}_j + \frac{Y_{ij}}{2} z_j \overline{z}_i\right) ~+~ C_{i0} 
$$
$$
\text{s.t.} ~~~
\left\{
\begin{array}{c}
(V_i^{\text{min}})^2 \leqslant |z_i|^2 \leqslant (V_i^{\text{max}})^2, ~~~ i=1,\hdots,n  \\[.2cm]
P_i^{\text{min}} - P^\text{dem}_i \leqslant ~  \sum\limits_{j=1}^n \frac{\overline{Y}_{ij}}{2} z_i \overline{z}_j + \frac{Y_{ij}}{2} z_j \overline{z}_i ~ \leqslant P_i^{\text{max}} - P^\text{dem}_i , ~~~ i=1,\hdots,n  \\[.4cm]
Q_i^{\text{min}} - Q^\text{dem}_i \leqslant ~ \sum\limits_{j=1}^n \frac{\overline{Y}_{ij}}{2\textbf{i}} z_i \overline{z}_j - \frac{Y_{ij}}{2\textbf{i}} z_j \overline{z}_i ~ \leqslant Q_i^{\text{max}} - Q^\text{dem}_i, ~~~ i=1,\hdots,n \\[.4cm]
\left| B_{ij} z_i \bar{z}_i + Y_{ij} z_j \bar{z}_i \right|^2  \leqslant (S_{ij}^{\text{max}})^2 , ~~~ i,j=1,\hdots,n,~\text{when}~ Y_{ij} \neq 0
\end{array}
\right.
$$
\textit{where all symbols in capital letters are physical constants.}
\textit{Figure} \ref{fig:opf} \textit{illustrates a global solution on an instance with} 14 \textit{complex variables} \cite[IEEE 14 Bus]{ieee_test_cases}\textit{. Each variable corresponds to a node in the graph and represents the voltage at that node. Once the voltages are computed, one can deduce the power production at each generator and the power flows on the edges. In order to supply} 261 MW \textit{of power to consumers (in red), the least expensive generation plan entails a total power production of} 268 MW \textit{(in black). The global solution was computed with the first order relaxation of the real Lasserre hierarchy (after converting the problem to real numbers). This is in accordance with the seminal work of Lavaei and Low} \cite{lavaei-low-2012} \textit{who showed that the first order relaxation solves many instances of the optimal power flow problem. It was later shown that there are also many instances that need higher-order relaxations} \cite{borden-demarco-lesieutre-molzahn-2011}\textit{. Such instances can be found in Table} \ref{tab:results}.
\begin{figure}[!h]
	\centering
	\includegraphics[width=.9\textwidth]{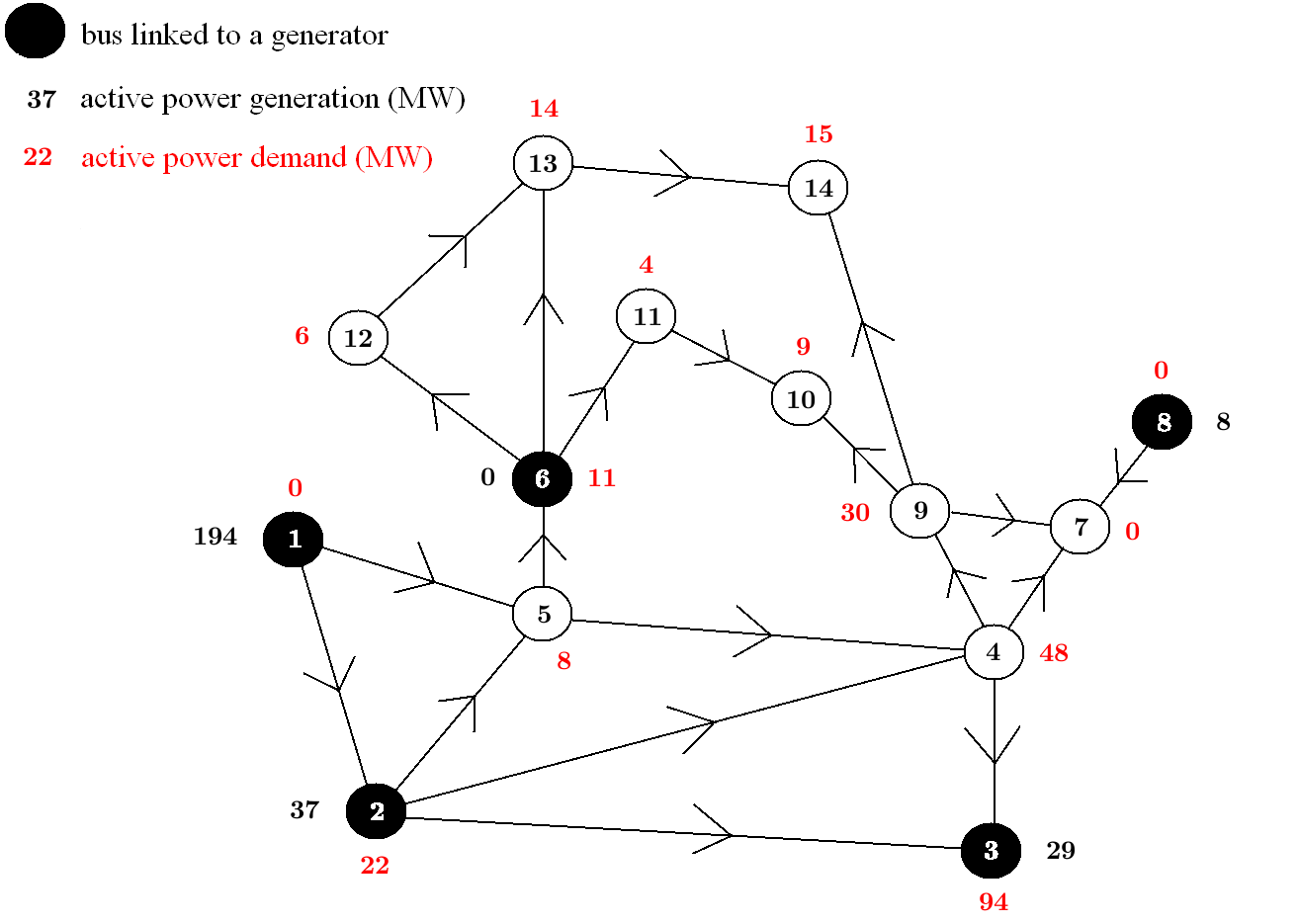}
	\caption{Globally optimal power flow in the Midwestern United States}
	 \label{fig:opf}
\end{figure}
\end{example}

In order to solve complex polynomial optimization problems, we follow the point of view of Lasserre \cite{lasserre-2000,lasserre-2001} in real numbers, that is the reformulation
\begin{equation}
\inf_{\mu \in \mathcal{M}_+(K) } \int_K f d\mu ~~~ \text{subject to} ~~~ \int_K d\mu = 1 
\end{equation}
where $\mathcal{M}_+(K)$ denotes the set of finite positive Borel measures on $K$. Lasserre observes that if the objective and constraints are real polynomials, then one may invoke the real moments $\int_K x^{\alpha} d\mu, ~ \alpha \in \mathbb{N}^n,$ of the measure $\mu$. We remark that with complex polynomials, this leads instead to the complex moments of the measure $\mu$, that is 
\begin{equation}
\int_K z^\alpha \bar{z}^\beta d\mu ~,~~~ \forall \alpha,\beta \in \mathbb{N}^n.
\end{equation}
Complex moments, like real moments, characterize the measure when $K$ is compact, thanks to the Stone-Weiestrass theorem. Note that when $K$ is compact, Borel measures are referred to as Radon measures and identify with the topological dual of the continuous functions from $K$ to $\mathbb{R}$ equipped with the operator norm. This is due to the Riesz representation theorem (see standard textbooks, e.g. \cite{rudin-1987}).
  
In order to define the original Lasserre hierarchy, the sequence of moments is truncated $\left\{ \int_K x^{\alpha} d\mu , ~  |\alpha| \leqslant 2d \right\}$ where $d$ is the truncation order and $|\alpha| := \sum_{k=1}^n \alpha_k$. In order to define the complex Lasserre hierarchy, we suggest truncating as follows: $\left\{ \int_K z^\alpha \bar{z}^\beta d\mu,~ |\alpha|, |\beta| \leqslant d \right\}$.
This naturally leads to a moment/sum-of-squares hierarchy in complex numbers:
$$
\boxed{
\begin{array}{lcl}
\inf_y ~ L_y(f) & \text{s.t.}& y_{0,0} = 1,~~  M_d(y) \succcurlyeq 0, ~~\text{and}~~ M_{d-k_i}(g_iy) \succcurlyeq 0, ~ i = 1, \hdots, m \\\\
\sup_{\lambda, \sigma} ~ \lambda & \text{s.t.} & f - \lambda = \sigma_0 + \sigma_1 g_1 + \hdots + \sigma_m g_m
\end{array}}
$$
where $\succcurlyeq$ stand for positive semidefinite. It relies on the following key notions:\\
\begin{itemize}
\item The \textit{complex moment matrix} is a Hermitian matrix defined by 
\begin{equation} 
M_d(y) := (y_{\alpha,\beta})_{|\alpha|,|\beta|\leqslant d}
\end{equation}
In contrast to the real moment matrix (in the original Lasserre hierarchy), it is \textit{not} a Hankel matrix. In other words, $y_{\alpha,\beta}$ is not necessarily only a function of $\alpha+\beta$. In the real case, we have $x^{\alpha} x^{\beta} = x^{\alpha+\beta}$ for $x\in \mathbb{R}^n$, whereas in the complex case, no such relationship holds for $z^\alpha \bar{z}^\beta$ where $z\in \mathbb{C}^n$.\footnote{In fact, if one were to enforce the Hankel property in the complex moment matrix, one obtains the real Lasserre hierarchy applied to the complex polynomial optimization problem where all the complex variables are restrained to the real line. But make no confusion: this is \textit{not} the real polynomial optimization problem obtained by identifying real and imaginary parts of the complex variables.}\\
\item The \textit{Riesz functional} is defined by
\begin{equation}
L_y(f) ~:=~ \sum_{\alpha,\beta} ~ f_{\alpha,\beta} ~  y_{\alpha,\beta}
\end{equation}
\item The \textit{localizing matrices} are defined by 
\begin{equation}
M_{d-k_i}(g_iy) := \left(\sum_{\gamma,\delta} ~ g_{i,\gamma,\delta} ~ y_{\alpha+\gamma,\beta+\delta}\right)_{|\alpha|,|\beta|\leqslant d-k_i}
\end{equation}
where $k_i := \max \{ |\alpha|,|\beta| ~\text{s.t.}~ g_{i,\alpha,\beta} \neq 0 \}$. Naturally, the truncation order $d$ must be greater than or equal to $d^{\text{min}} := \max \{k_0,k_1,\hdots,k_m\}$ where $k_0 := \max \{ |\alpha|,|\beta| ~\text{s.t.}~ f_{\alpha,\beta} \neq 0 \}$.\\
\item A polynomial $\sigma(z,\bar{z}) = \sum_{|\alpha|,|\beta|\leqslant d} \sigma_{\alpha,\beta} z^\alpha \bar{z}^\beta$ is a \textit{Hermitian sum-of-squares}, i.e. it belongs to $\Sigma_d[z,\bar{z}]$, if it is of the form\footnote{We use $|\cdot|$ to denote the modulus of a complex number. In a Hermitian sum-of-squares, the dependence on both $z$ and $\bar{z}$ can be seen upon developing the squares.}
\begin{equation}
\sigma(z,\bar{z}) ~ = ~ \sum_k \left| \sum_{|\alpha|\leqslant d} p_{k,\alpha} z^\alpha \right|^2  ~~\text{where}~~ p_{k,\alpha} \in \mathbb{C}.
\end{equation}
This is equivalent to $(\sigma_{\alpha,\beta})_{|\alpha|,|\beta|\leqslant d} \succcurlyeq 0$ where $\alpha,\beta \in \mathbb{N}^n$. In the complex Lasserre hierarchy, $\sigma_0 \in \Sigma_{d}[z,\bar{z}]$ and $\sigma_i \in \Sigma_{d-k_i}[z,\bar{z}], ~ i = 1, \hdots, m$.\\
\item A Hermitian sum-of-squares is a special case of a \textit{real sum-of-squares} (used in the original Lasserre hierarchy), that is, a polynomial of the form
\begin{equation}
\sigma(z,\bar{z}) ~ = ~ \sum_k \left| \sum_{|\alpha+\beta|\leqslant d} p_{k,\alpha,\beta} z^\alpha \bar{z}^{\beta} \right|^2  ~~\text{where}~~ p_{k,\alpha,\beta} \in \mathbb{C}.
\end{equation}
This is equivalent to the existence of a real positive semidefinite matrix $(\varphi_{\alpha,\beta})_{|\alpha|,|\beta|\leqslant d}$ where $\alpha,\beta \in \mathbb{N}^{2n}$ such that $\sigma(z,\bar{z}) = \sum_{\alpha,\beta} \varphi_{\alpha,\beta} x^{\alpha+\beta} $. (We have identified real and imaginary parts $z_k := x_k +  x_{k+n}\textbf{i}$.)\\
\end{itemize}

\begin{example}
\normalfont
$x_1^2+ 2x_1+1+x_2^2$ \textit{with} $x_1,x_2 \in \mathbb{R}$ \textit{is a Hermitian sum-of-squares because it is equal to} $|1+x_1+x_2\textbf{i}|^2 = |1+z|^2$ \textit{where} $z := x_1+ x_2\textbf{i}$\textit{. In contrast,} $x_1^2+ 2x_1 + 1$ \textit{is a real sum-of-squares but not a Hermitian sum-of-squares. Indeed, we have} $x_1^2+ 2x_1+ 1  = \left|1+\frac{1}{2}z+\frac{1}{2}\bar{z}\right|^2 = 1 + z + \bar{z} + \frac{1}{4} z^2 + \frac{1}{2}|z|^2 + \frac{1}{4} \bar{z}^2$\textit{. In other words}
\begin{equation}
x_1^2+ 2x_1+ 1 ~~ = ~
\begin{pmatrix} 
\hphantom{.}1\hphantom{^2} \\
\hphantom{.}z\hphantom{^2} \\
\hphantom{.}z^2
\end{pmatrix}^*
\begin{pmatrix}
1 & 1 & 1/4 \\
1 & 1/2 & 0 \\
1/4 & 0 & 0
\end{pmatrix}
\begin{pmatrix} 
\hphantom{.}1\hphantom{^2} \\
\hphantom{.}z\hphantom{^2} \\
\hphantom{.}z^2
\end{pmatrix}
\end{equation}
\textit{where} $(\cdot)^*$ \textit{stands for conjugate transpose. The above matrix is unique and it is not positive semidefinite. Hence the polynomial is not a Hermitian sum-of-squares. The unicity in the Hermitian decomposition} (\textit{which is true for any polynomial, not just in this example}) \textit{contrasts with the non-unicity in the real decomposition} $x_1^2+ 2x_1+ 1  = \hdots $
$$
\scriptsize
\begin{pmatrix} 
1 \\
x_1 \\
x_2 \\
x_1^2 \\
x_1 x_2 \\
x_2^2
\end{pmatrix}^T
\begin{pmatrix}
1 & 1 & 0 & 0 & 0 & 0 \\
1 & 1 & 0 & 0 & 0 & 0 \\
0 & 0 & 0 & 0 & 0 & 0 \\
0 & 0 & 0 & 0 & 0 & 0 \\
0 & 0 & 0 & 0 & 0 & 0 \\
0 & 0 & 0 & 0 & 0 & 0 
\end{pmatrix}
\begin{pmatrix} 
1 \\
x_1 \\
x_2 \\
x_1^2 \\
x_1 x_2 \\
x_2^2
\end{pmatrix}
=
\begin{pmatrix} 
1 \\
x_1 \\
x_2 \\
x_1^2 \\
x_1 x_2 \\
x_2^2
\end{pmatrix}^T
\begin{pmatrix}
1 & 1 & 0 & 1/2 & 0 & 0 \\
1 & 0 & 0 & 0 & 0 & 0 \\
0 & 0 & 0 & 0 & 0 & 0 \\
1/2 & 0 & 0 & 0 & 0 & 0 \\
0 & 0 & 0 & 0 & 0 & 0 \\
0 & 0 & 0 & 0 & 0 & 0 
\end{pmatrix}
\begin{pmatrix} 
1 \\
x_1 \\
x_2 \\
x_1^2 \\
x_1 x_2 \\
x_2^2
\end{pmatrix}
$$
\textit{where} $(\cdot)^T$ \textit{stands for transpose. One of the above matrices is positive semidefinite, making the polynomial a real sum-of-squares} (\textit{which is otherwise obvious}).
\end{example}
\vspace{.3cm}
\begin{example}
\normalfont
\textit{Consider the following complex polynomial optimization problem}
\begin{equation}
\inf_{z \in \mathbb{C}} ~~ z+\bar{z} ~~~\text{s.t.}~~~ |z|^2 = 1 
\end{equation}
\textit{whose optimal value is} $-2$\textit{. Letting} $z =: x_1 + \textbf{i} x_2$\textit{, it can converted into real numbers:}
\begin{equation}
\inf_{x_1,x_2 \in \mathbb{R}} ~~ 2x_1 ~~~\text{s.t.}~~~ x_1^2 + x_2^2 = 1.
\end{equation}
\textit{It can be solved to global optimality using real sum-of-squares since}
\begin{equation}
2x_1 ~ - ~ (-2) ~~=~~ 1 + (x_1+x_2)^2 ~~+~~ 1 \times (1-x_1^2 - x_2^2). \end{equation} 
\textit{But it can also be solved using Hermitian sum-of-squares since}
\begin{equation}
z+\bar{z} ~~ - ~ (-2) ~~=~~~~~~~ |1+z|^2 ~~~~~~+~~~~ 1 \times (1-|z|^2).~~~~~~~~~~~~~~ 
\end{equation} 
\end{example}
\indent Hermitian sum-of-squares entail a trade-off. At each truncation order, they are cheaper to compute, but they potentially provide a relaxation bound of poorer quality. More precisely, as the number of variables grows, the moment matrix in the real Lasserre hierarchy is $2^d$ times bigger than the moment matrix in the complex hierarchy. Regarding the optimal power flow problem, the relaxation bounds for the real and complex hierarchies are the same at each order in all our numerical experiments.

\begin{example}
\normalfont
\textit{The advantage of Hermitian sum-of-squares for finding global minimizers to the optimal power flow problem can be seen in} Table \ref{tab:results}\textit{. In} 17 \textit{out of the} 18 \textit{instances, they are faster, sometimes up to an order of magnitude. The minimizers obtained are feasible up to} 0.005 p.u. \textit{at voltage constraints and} 1 MVA \textit{at all other constraints}\footnote{Typical violations are smaller than 1~MVA. For instance, with the complex hierarchy PL-3012wp has over 99\% of the buses with less than 0.02~MVA violation, and only 0.09\% of the buses with greater than 0.1~MVA violation. Maximum line flow violation is 0.0006~MVA.}\textit{, and the objective evaluated in the minimizers matches the relaxation bound with} 0.05\% \textit{relative to the bound. In order to obtain these results, sparsity is exploited using the multi-ordered Lasserre hierarchy} (Section \ref{sec:Multi-ordered Lasserre hierarchy})\textit{, and symmetry is exploited using the block diagonal Lasserre hierarchy} (Section \ref{sec:Block diagonal Lasserre hierarchy})\textit{. We thus increment the relaxation order at up to} 176 \textit{constraints, and up to order} 2\textit{. The largest maximal clique size is} 19\textit{. Regarding the software,} YALMIP \mbox{2015.06.26}~\cite{yalmip} \textit{and} MOSEK \textit{are used for the experiments. For test case descriptions, see}~\cite{mh_sparse_msdp} \textit{for} \mbox{case14Q}--\mbox{case300}, \cite{nesta} \textit{for the ``nesta'' cases with ``active power increases''} (API) \textit{ loading scenarios, and} \cite{murillosanchez-thomas-zimmerman-2011,josz-2016,pegase}  \textit{for} \mbox{PL-2383wp}--\mbox{PEGASE-2869}\textit{. For the Polish} (PL) \textit{and} \mbox{PEGASE} \textit{cases, a preprocessing step was used to eliminate lines with impedances less than} \mbox{$1\times 10^{-3}$} \textit{and} \mbox{$3\times 10^{-3}$} \textit{per unit, respectively, a} \mbox{$1\times 10^{-4}$} \textit{per unit minimum resistance was enforced on each line} (\textit{as in} \cite{lavaei-low-2012})\textit{, and the objective was active power loss minimization.} Table \ref{tab:results} \textit{displays the number of variables and constraints after the preprocessing step.}
\begin{table}[!ht]
\centering
\caption{Global value found by multi-ordered Lasserre hierarchy and solver time in seconds}
\begin{tabular}{|l|c|c|c|c|c|c|}
\hline
\multicolumn{1}{|c|}{\textbf{Case}} & \multicolumn{1}{c|}{\textbf{Real}} & \multicolumn{1}{c|}{\textbf{Const-}} & \multicolumn{2}{c|}{\textbf{Real Lasserre}} & \multicolumn{2}{c|}{\textbf{Complex Lass.}}\\\cline{4-7}
\multicolumn{1}{|c|}{\textbf{Name}}  & \multicolumn{1}{c|}{\textbf{Variables}} & \multicolumn{1}{c|}{\textbf{raints}} & \hspace*{1em}\textbf{Obj.}\hspace*{1em} & \textbf{Time} & \hspace*{1em}\textbf{Obj.}\hspace*{1em} & \textbf{Time} \\
\hline
case14Q & \hphantom{1,1}28 & \hphantom{11,1}57      & \hphantom{11}3,302 & \hphantom{1,11}4.7 & \hphantom{11}3,302 & \hphantom{1,11}2.6 \\
case14L & \hphantom{1,1}28 & \hphantom{11,1}97      & \hphantom{11}9,359 & \hphantom{1,11}1.9 & \hphantom{11}9,359 & \hphantom{1,11}1.5 \\
case39Q  & \hphantom{1,1}78 & \hphantom{11,}239     & \hphantom{1}11,221 & \hphantom{1,}741.7 & \hphantom{1}11,211 & \hphantom{1,1}58.7 \\
case39L  & \hphantom{1,1}78 & \hphantom{11,}239    & \hphantom{1}41,921 & \hphantom{1,11}2.3 & \hphantom{1}41,921 & \hphantom{1,11}1.4 \\
case57Q   & \hphantom{1,}114 & \hphantom{11,}192    & \hphantom{11}7,352 & \hphantom{1,11}3.4 & \hphantom{11}7,352 & \hphantom{1,11}3.3 \\
case57L   & \hphantom{1,}114 & \hphantom{11,}352    & \hphantom{1}43,984 & \hphantom{1,11}1.4 & \hphantom{1}43,984 & \hphantom{1,11}1.3 \\
case118Q    & \hphantom{1,}236 & \hphantom{11,}516     & \hphantom{1}81,515 & \hphantom{1,1}15.7 & \hphantom{1}81,515 & \hphantom{1,11}3.7 \\
case118L  & \hphantom{1,}236 & \hphantom{11,}888   & 134,907 & \hphantom{1,1}10.5 & 134,907 & \hphantom{1,11}5.5 \\
case300     & \hphantom{1,}600 & \hphantom{1}1,107   & 720,040 & \hphantom{1,11}7.2 & 720,040 & \hphantom{1,11}4.1 \\
nesta\_case24       & \hphantom{1,1}48 & \hphantom{11,}526   & \hphantom{11}6,421 & \hphantom{1,}246.1 & \hphantom{11}6,421 & \hphantom{1,1}61.7\\
nesta\_case30     & \hphantom{1,1}60 & \hphantom{11,}272  & \hphantom{111,}372 & \hphantom{1,}302.7 & \hphantom{111,}372 & \hphantom{1,1}15.4 \\
nesta\_case73       & \hphantom{1,}146 & \hphantom{1}1,605   & \hphantom{1}20,125 & \hphantom{1,}506.9 & \hphantom{1}20,124 & \hphantom{1,1}52.6 \\
PL-2383wp     & 4,354 & 12,844  & \hphantom{1}24,990 & \hphantom{1,}583.4 & \hphantom{1}24,991 & \hphantom{1,1}53.9 \\
PL-2746wop    & 4,378 & 13,953  & \hphantom{1}19,210 & 2,662.4 & \hphantom{1}19,212 & \hphantom{1,}124.3 \\
PL-3012wp    & 4,584 & 14,455  & \hphantom{1}27,642 & \hphantom{1,}318.7 & \hphantom{1}27,644 & \hphantom{1,}141.0 \\ 
PL-3120sp    & 4,628 & 13,948  & \hphantom{1}21,512 & \hphantom{1,}386.6 & \hphantom{1}21,512 & \hphantom{1,}193.9 \\
PEGASE-1354 & 1,966 & \hphantom{1}6,444  & \hphantom{1}74,043 & \hphantom{1,}406.9 & \hphantom{1}74,042 & 1,132.6 \\ 
PEGASE-2869 & 4,240 & 12,804  & 133,944 & \hphantom{1,}921.3 & 133,939 & \hphantom{1,}700.8 \\ 
\hline
\end{tabular}
\label{tab:results}
\end{table}
\end{example}

Having motivated the introduction of the complex Lasserre hierarchy, we discuss asymptotic convergence and finite convergence in the next two sections. These two aspects are significantly different from the real hierarchy. Most other aspects of the real hierarchy carry over to the complex hierarchy in a straightforward fashion, including strong duality \cite{josz-2015} and the generalized Lagrangian interperation (see \cite[Section 7]{josz-phd} for details). One aspect that is unresolved is the question of generic finite convergence \cite{nie-2014}, which is a subject for future research.

\section{Asymptotic convergence}
\label{sec:Asymptotic convergence}
In 1968, Quillen~\cite{quillen-1968} showed that a real-valued bihomogeneous complex polynomial that is positive away from the origin can be decomposed as a Hermitian sum-of-squares when it is multiplied by $(|z_1|^2 + \hdots + |z_{n}|^2)^r$ for some $r\in \mathbb{N}$. The result was rediscovered by Catlin and D'Angelo~\cite{catlin-1996} and ignited a search for complex analogues of Hilbert's seventeenth problem~\cite{angelo-2002,angelo-2010} and the ensuing Positivstellens\"atze~\cite{putinar-2006,putinar-2012,putinar-scheiderer-2012,putinar-2013}. Notably, D'Angelo and Putinar proved the following powerful result in 2008.

\begin{theorem}[D'Angelo's and Putinar's Positivstellenstatz~\cite{angelo-2008}]
\label{th:angelo} Assume that one of the constraints of $K$ is a sphere $|z_1|^2 + \hdots + |z_n|^2 = R^2$ for some radius $R > 0$. If $f>0$ on $K$, then there exists Hermitian sum-of-squares $\sigma_0, \hdots, \sigma_m$ such that 
\begin{equation}
f = \sigma_0 + \sum_{i=1}^m \sigma_i g_i.
\end{equation}
\end{theorem}
This theorem naturally admits a dual perspective.
\begin{theorem}[Putinar and Scheiderer~\cite{putinar-2013}]
\normalfont
\textit{If one of the constraints of} $K$ \textit{is a sphere, then the following properties are equivalent:}
\begin{enumerate}
\item $\exists \mu \in \mathcal{M}_+(K): ~\forall \alpha,\beta \in \mathbb{N}^n, ~~~ y_{\alpha,\beta} = \int_{K} z^\alpha \bar{z}^\beta d\mu$;
\item $\forall d \geqslant d^{\text{min}}~, ~~~  M_d(y) \succcurlyeq 0, ~  M_{d-k_i}(g_i y) \succcurlyeq 0$.
\end{enumerate}
\end{theorem}
Global convergence in the complex hierarchy is thus guaranteed in the presence of a sphere constraint. This is in contrast to the real hierarchy where it is guaranteed in the presence of a ball constraint. A sphere may appear more restrictive than a ball. However, this is sufficient to solve complex polynomial optimization problems with compact feasible sets. Indeed, one can add a slack variable $z_{n+1} \in \mathbb{C}$ and a redundant constraint $|z_1|^2 + \hdots + |z_{n+1}|^2 = R^2$ to the description of the feasible set when it is in a ball of radius $R$. This is similar to Lasserre who proposes to add a redundant ball constraint $x_1^2 + \hdots + x_n^2 \leqslant R^2$.

\begin{example}
\label{ex:angelo}
\normalfont
\textit{Consider the following optimization problem} 
\begin{equation}
 \inf_{z \in \mathbb{C}} ~~~ 1-\frac{4}{3}|z|^2+\frac{7}{18}|z^2|^2 ~~~~~\text{s.t.}~~~~~  1 - |z|^2 \geqslant 0
\end{equation}
\textit{whose optimal value is} $1/18$\textit{. D'Angelo and Putinar}~\cite{angelo-2008} \textit{have demonstrated that there does not exist Hermitian sum-of-squares} $\sigma_0$ \textit{and} $\sigma_1$ \textit{such that}
\begin{equation}
1-\frac{4}{3}|z|^2+\frac{7}{18}|z|^4 ~ = ~ \sigma_0(z,\bar{z}) + \sigma_1(z,\bar{z}) ( 1 - |z|^2 ).
\end{equation}
\textit{As a result, the complex hierarchy cannot exceed the value} 0\textit{. In fact, it finds} $-1/3$ \textit{at all orders because}
\begin{equation}
M(y) = 
\begin{array}{cccccc}
               & \hphantom{^2}1\hphantom{^2} & \bar{z} & \hphantom{^2}\bar{z}^2 & \hphantom{^2}\bar{z}^3 &    \\[.1cm]
1\hphantom{^2}             & 1 & 0 & 0    & 0     & \hdots \\[.1cm]
z\hphantom{^2}     & 0 & 1 & 0    & 0    &    \\[.1cm]
z^2 & 0 & 0 & 0    & 0    &  \\[.1cm]
z^3 & 0 & 0 & 0    & 0    &    \\[.1cm]
       & \vdots &    &       &        & \ddots
\end{array}
\end{equation}
\textit{is a primal optimal point. We propose to add a complex slack variable}
\begin{equation}
 \inf_{z_1,z_2 \in \mathbb{C}} ~~~ 1-\frac{4}{3}|z_1|^2+\frac{7}{18}|z_1|^4 ~~~~~\text{s.t.}~~~~~  1 - |z_1|^2 - |z_2|^2 = 0,
\end{equation}
\textit{enabling the second order complex relaxation to find the global infimum}
\begin{equation}
\begin{array}{c}
1-\frac{4}{3}|z_1|^2+\frac{7}{18}|z_1|^4 - \frac{1}{18} \\[.2cm]
=\\[.2cm]
 \frac{5}{18} |z_2|^2 +  \frac{5}{18} |z_1 z_2|^2 + \frac{2}{3} |z_2|^4\\[.2cm]
+\\[.2cm]
\left(\frac{17}{18} -  \frac{7}{18} |z_1|^2 + \frac{2}{3} |z_2|^2\right) (1 - |z_1|^2 - |z_2|^2).
\end{array}
\end{equation}
\textit{Note that the polynomial that multiplies the constraint is not a Hermitian sum-of-squares. This would be a contradiction when taking} $z_2 = 0$.
\end{example}

We next discuss a weaker condition ensuring global convergence in the Lasserre hierarchy.
In the real hierarchy, convergence is guaranteed if the Archimedean condition holds, that is to say, if there exists $R>0$ and real sums-of-squares $\sigma_0,\hdots,\sigma_m$ such that $
R^2-x_1^2-\hdots-x_n^2 = \sigma_0(x) + \sum_{i=1}^m \sigma_i(x) g_i(x) ,~ \forall x\in \mathbb{R}^n$. In the complex hierarchy, a similar condition can be deduced from the work of Putinar and Scheiderer~\cite[Propositions 6.6 and 3.2 (iii)]{putinar-2013}. For notational convenience, suppose that some of the inequality constraints $g_i(z,\bar{z}) \geqslant 0$ are actually equality constraints $g_i(z,\bar{z}) = 0$. Let $E \subset \{1,\hdots,m\}$ denote the indices of the equality constraints. Global convergence in the complex hierarchy is guaranteed if there exists $R>0$, a Hermitian sum-of-squares $\sigma_0$, and real-valued complex polynomials $p_i$'s such that $R^2-|z_1|^2-\hdots-|z_n|^2 = \sigma_0(z,\bar{z}) + \sum_{i\in E} p_i(z,\bar{z}) g_i(z,\bar{z}) ,~ \forall z\in \mathbb{C}^n$.
In particular, in presence of the equalities $|z_k|^2 = 1, ~ k = 1,\hdots,n$, convergence is guaranteed. In this case, there is no need to add a slack variable as suggested above. This applies for instance to the non-bipartite Grothendieck problem over the complex numbers~\cite{bandeira-2014}. Interestingly, in the optimal power flow problem, despite the absence of such equalities, global convergence is attained in all numerical experiments without adding a slack variable (as reported in Table \ref{tab:results}). Of course, it is also attained when adding a slack variable.

When the weaker assumption presented above does not hold, there exists a way to quantify how far D'Angelo's and Putinar's Positivstellensatz is from being true. This is given by the Hermitian complexity \cite{putinar-2012} of the ideal associated with the equality constraints. This number is related to the greatest number of distinct points (possibly infinite) $z^{(i)} \in \mathbb{C}^n,~ 1 \leqslant i \leqslant p$, such that $g_k(z^{(i)},z^{(j)}) = 0$ for all $k \in E$. Loosely speaking, the greater this number, the farther away the Positivstellensatz is from being true. In particular, when one of the equalities is $\sigma(z,\bar{z}) + |z_1|^2+\hdots+|z_n|^2 = R^2$ with $\sigma$ a Hermitian sum-of-squares and $R>0$, then the Hermitian complexity is equal to 1. The Positivstellensatz is then true, in accordance with the weaker assumption presented above. 

\section{Finite convergence}
\label{sec:Finite convergence}
The relaxation of order $d$ of the complex Lasserre hierarchy yields a set of complex numbers $(y_{\alpha,\beta})_{|\alpha|,|\beta| \leqslant d}$. Global solutions may be extracted if there exists a measure that represents those $y_{\alpha,\beta}$ that appear in the objective and constraint functions. In particular, this is true if there exists positive Borel measure $\mu$ supported on the semi-algebraic set $K$ such that $y_{\alpha,\beta} = \int_{\mathbb{C}^n} z^\alpha \bar{z}^\beta d\mu$ for all $|\alpha|,|\beta| \leqslant d^\text{min}$.
In this case, the complex hierarchy has finite convergence and global optimality is attained. In this section, we will show that the conditions ensuring finite convergence in the complex hierarchy differ significantly from the real hierarchy. Additional conditions need to be satisfied as shown in the next result. We will use the definition $d_K := \max \{2,k_1,\hdots,k_m\}$ where the number 2 will be explained in the next section. It is meant to guarantee that $d_K \geqslant 2$, in contrast to the real case when it is only required that $d_K \geqslant 1$ (see \cite[equation (6.1)]{laurent2009}).
\begin{proposition}
\label{prop:finite}
\normalfont
\textit{Assume that one of the constraints of the multivariate} ($n>1$) \textit{optimization problem is a ball} $|z_1|^2 + \hdots + |z_{n}|^2 \leqslant R^2$ \textit{for some radius} $R > 0$\textit{. Consider an optimal solution $y$ to the complex moment relaxation of order} 
$d$.
\vspace{.1cm}

$\bullet$ \textit{If there is an integer} $t$ \textit{such that} $d^\text{min} \leqslant t \leqslant d$ \textit{and} $\text{rank} M_{t} (y) = 1$\textit{, then global optimality is attained and there is at least one global solution.}
\vspace{.1cm}

$\bullet$ \textit{If there is an integer} $t$ \textit{such that} $\max \{ d^\text{min} , d_K \} \leqslant t \leqslant d$ \textit{and if the following conditions hold:}
\vspace{.2cm}

$
\begin{array}{l}
\text{1.} ~~ 
\text{rank} ~ M_{t} (y) = \text{rank} ~ M_{t-d_K} (y) ~ (=: S)\\[.3cm]
\text{2.} ~~ 
\begin{pmatrix} 
M_{t-d_K}(y) & M_{t-d_K}(z_i y) & M_{t-d_K}(z_j y) \\ 
M_{t-d_K}(\bar{z}_i y) & M_{t-d_K}(|z_i|^2 y)  & M_{t-d_K}(z_j \bar{z}_i y) \\
M_{t-d_K}(\bar{z}_j y) & M_{t-d_K}(z_i \bar{z}_j y) & M_{t-d_K}(|z_j|^2 y)
\end{pmatrix} 
\succcurlyeq 0 , ~ \forall 1 \leqslant i < j \leqslant n
\end{array}
$\\\\
\textit{then global optimality is attained and there are at least at least $S$ global solutions.} 
\end{proposition}
\begin{proof}
This is a consequence of Theorem \ref{th:hypo}.
\end{proof}
\vspace{.1cm}

Under the assumptions of Proposition \ref{prop:finite}, if the rank is equal to one, then a global solution $z$ can then be read from the moment matrix, i.e. $z = (y_{\alpha,0})_{|\alpha|=1} \in \mathbb{C}^n$. This is just like in the real Lasserre hierarchy. Otherwise, if the rank is greater than one ($S>1$), then $S$ global solutions can be extracted using \cite[Algorithm 4.1]{harmouch2017}. In fact, this algorithm can also extract global solutions from the real Lasserre hierarchy. It appears to be the most efficient way to do so as it only requires one singular value decomposition followed by an eigendecomposition. Earlier approaches can be found in \cite{henrion2005,henrion2009,laurent2009}.

\begin{example}
\label{eg:ellipse}
\normalfont
\textit{Consider the following problem whose elliptic constraint is taken from} \cite{putinar-scheiderer-2012}\textit{:}
\begin{equation}
\inf\limits_{z_1,z_2 \in \mathbb{C}} ~~~ 3 - |z_1|^2 - \frac{1}{2} \textbf{i}z_1\bar{z}_2^2 + \frac{1}{2} \textbf{i}z_2^2\bar{z}_1
\end{equation}
\begin{equation}
\text{s.t.} ~~~~~~ 
\left\{
\begin{array}{l}
|z_1|^2 - \frac{1}{4} z_1^2 - \frac{1}{4} \bar{z}_1^2 = 1 \\[.3cm]  |z_1|^2 + |z_2|^2 = 3 \\[.3cm]
\textbf{i} z_2 - \textbf{i} \overline{z}_2 = 0, ~~~ z_2 + \overline{z}_2 \geqslant 0.
\end{array}
\right.
\end{equation}
\textit{The feasible set is represented in Figure} \ref{fig:ellipse}\textit{, which we generated using} POV-Ray 3.7.0 \cite{povray2013}.
\begin{figure}[!h]
	\centering
	\includegraphics[width=.6\textwidth]{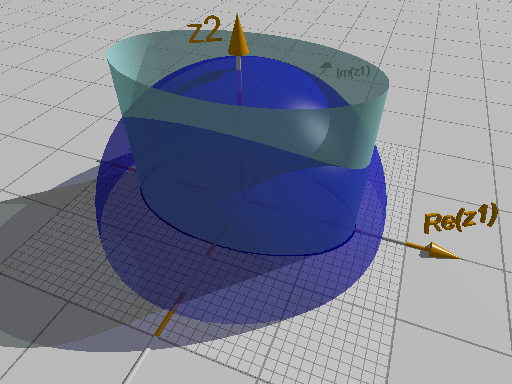}
	\caption{Feasible set at the intersection of the sphere and the elliptic cylinder}
	 \label{fig:ellipse}
\end{figure}
\textit{The hierarchy starts at the second order} ($d^{\text{min}} = d_K = 2$) \textit{which yields}\footnote{MATLAB 2015b, CVX 2.0 \cite{cvx}, and SDPT3 4.0 \cite{toh2012} are used for the numerical experiments.} \textit{the lower bound} 0.155089 \textit{and the optimal moment matrix}
$$
M_2(y) = 
\begin{array}{ccccccc}
 & 1 & \bar{z}_1 & \bar{z}_2 & \bar{z}_1^2 & \bar{z}_1 \bar{z}_2 & \bar{z}_2^2 \\[.25em]
1 & \hphantom{-}1.0000\hphantom{\textbf{i}} & \hphantom{-}0.3747\textbf{i}  & \hphantom{-}0.8485\hphantom{\textbf{i}} & \hphantom{-}1.8272\hphantom{\textbf{i}} & \hphantom{-}0.5100\textbf{i} &  \hphantom{-}1.0864\hphantom{\textbf{i}}  \\[.25em]
z_1 & -0.3747\textbf{i} &  \hphantom{-}1.9136\hphantom{\textbf{i}} & -0.5100\textbf{i}  & \hphantom{-}0.1929\textbf{i} & \hphantom{-}1.0505\hphantom{\textbf{i}}  & -0.9313\textbf{i} \\[.25em]
z_2 & \hphantom{-}0.8485\hphantom{\textbf{i}} & \hphantom{-}0.5100\textbf{i}  & \hphantom{-}1.0864\hphantom{\textbf{i}} &  \hphantom{-}0.9245\hphantom{\textbf{i}} &  \hphantom{-}0.9313\textbf{i}  & \hphantom{-}1.4950\hphantom{\textbf{i}} \\[.25em]
z_1^2 & \hphantom{-}1.8272\hphantom{\textbf{i}} & -0.1929\textbf{i} &  \hphantom{-}0.9245\hphantom{\textbf{i}} & \hphantom{-}4.5886\hphantom{\textbf{i}} & -0.1162\textbf{i} &  \hphantom{-}0.9324\hphantom{\textbf{i}} \\[.25em]
z_1 z_2 & -0.5100\textbf{i}  & \hphantom{-}1.0505\hphantom{\textbf{i}} & -0.9313\textbf{i}  & \hphantom{-}0.1162\textbf{i}  & \hphantom{-}1.1523\hphantom{\textbf{i}} & -1.4140\textbf{i} \\[.25em]
z_2^2 & \hphantom{-}1.0864\hphantom{\textbf{i}} & \hphantom{-}0.9313\textbf{i} &  \hphantom{-}1.4950\hphantom{\textbf{i}} & \hphantom{-}0.9324\hphantom{\textbf{i}} & \hphantom{-}1.4140\textbf{i}  & \hphantom{-}2.1069\hphantom{\textbf{i}}
\end{array}
$$
\textit{It holds that} $\text{rank}M_0(y) = 1$, $\text{rank}M_1(y) = 3$, \textit{and} $\text{rank}M_2(y) = 3$. \textit{Since} $\text{rank}M_0(y) \neq \text{rank}M_2(y)$\textit{, the rank condition in Proposition \ref{prop:finite} does not hold. The positive semidefine condition holds with} $t = 2$\textit{ but not with} $t=3$:
\begin{equation}
\text{sp}\left\{
\begin{pmatrix}
M_{1}(y) & M_{1}(z_1 y) & M_{1}(z_2 y) \\ 
M_{1}(\bar{z}_1 y) & M_{1}(|z_1|^2 y)  & M_{1}( z_2 \bar{z}_1 y) \\
M_{1}(\bar{z}_2 y) & M_{1}( z_1 \bar{z}_2 y) & M_{1}(|z_2|^2 y)
\end{pmatrix} 
\right\}
= 
\begin{pmatrix}
   -1.5874\\
   -0.1295\\
   -0.0000\\
    \hphantom{-}0.0000\\
    \hphantom{-}0.1574\\
    \hphantom{-}0.7711\\
    \hphantom{-}3.5471\\
    \hphantom{-}5.0544\\
    \hphantom{-}8.1869
\end{pmatrix}
\end{equation}
\textit{where} $\text{sp}\{\cdot \}$ \textit{stands for spectrum.} \textit{The third order complex relaxation yields the value} 0.428175 \textit{and the moment matrix satisfies} $\text{rank}M_3(y) = 1$\textit{. This yields a global solution} $(z_1,z_2) = (-0.8165\textbf{i} , 1.5275)$ \textit{which can be read from the third order moment matrix.} \textit{Interestingly, it is not necessary to go up to the third order relaxation. Since the positive semidefinite condition is a convex property, it can be added to the second order relaxation, with} $t=3$ \textit{for instance:}
\begin{equation}
\begin{pmatrix}
M_{1}(y) & M_{1}(z_1 y) & M_{1}(z_2 y) \\ 
M_{1}(\bar{z}_1 y) & M_{1}(|z_1|^2 y)  & M_{1}( z_2 \bar{z}_1 y) \\
M_{1}(\bar{z}_2 y) & M_{1}( z_1 \bar{z}_2 y) & M_{1}(|z_2|^2 y)
\end{pmatrix} \succcurlyeq  0.
\end{equation}
\textit{We then obtain the value} 0.428175 \textit{and the following moment matrix}
$$
\label{eq:enforce}
M_2(y) = 
\begin{array}{ccccccc}
 & 1 & \bar{z}_1 & \bar{z}_2 & \bar{z}_1^2 & \bar{z}_1 \bar{z}_2 & \bar{z}_2^2 \\[.25em]
1 &    \hphantom{-}1.0000\hphantom{\textbf{i}} & \hphantom{-}0.8165\textbf{i} &  \hphantom{-}1.5275\hphantom{\textbf{i}} & -0.6667\hphantom{\textbf{i}}    & \hphantom{-}1.2472\textbf{i} &  \hphantom{-}2.3333\hphantom{\textbf{i}} \\[.25em]
z_1 & -0.8165\textbf{i}  & \hphantom{-}0.6667\hphantom{\textbf{i}} & -1.2472\textbf{i}  & \hphantom{-}0.5443\textbf{i}   & \hphantom{-}1.0184\hphantom{\textbf{i}} & -1.9052\textbf{i}\\[.25em]
z_2 & \hphantom{-}1.5275\hphantom{\textbf{i}} & \hphantom{-}1.2472\textbf{i} &  \hphantom{-}2.3333\hphantom{\textbf{i}} & -1.0184\hphantom{\textbf{i}}  &  \hphantom{-}1.9052\textbf{i}  & \hphantom{-}3.5642\hphantom{\textbf{i}} \\[.25em]
z_1^2 & -0.6667\hphantom{\textbf{i}} & -0.5443\textbf{i} & -1.0184\hphantom{\textbf{i}} & \hphantom{-}0.4444\hphantom{\textbf{i}}    & -0.8315\textbf{i} & -1.5556\hphantom{\textbf{i}} \\[.25em]
z_1 z_2 &    -1.2472\textbf{i}  & \hphantom{-}1.0184\hphantom{\textbf{i}} & -1.9052\textbf{i}  & \hphantom{-}0.8315\textbf{i}   & \hphantom{-}1.5556\hphantom{\textbf{i}} &  -2.9102\textbf{i} \\[.25em]
z_2^2 &    \hphantom{-}2.3333\hphantom{\textbf{i}} & \hphantom{-}1.9052\textbf{i}  & \hphantom{-}3.5642\hphantom{\textbf{i}} &  -1.5556\hphantom{\textbf{i}}   & \hphantom{-}2.9102\textbf{i} &  \hphantom{-}5.4444\hphantom{\textbf{i}} 
\end{array}
$$
\textit{which satisfies} $\text{rank}M_2(y) = 1$\textit{. A global solution can be read in the first column:} $(z_1,z_2) = (-0.8165\textbf{i} , 1.5275)$.
\textit{We have just used a notion in operator theory to reduce the rank from} 3 \textit{to} 1 \textit{in a convex relaxation. For explanations, see the next section.}
\end{example}

\section{Truncated moment problem}
\label{sec:Truncated moment problem}

The complex Lasserre hierarchy brings into the picture a truncated moment problem which has not been considered in past literature to the best of our knowledge. Given a set of complex numbers $(y_{\alpha,\beta})_{|\alpha|,|\beta| \leqslant d}$, it raises the question of whether there exists positive Borel measure $\mu$ supported on the semi-algebraic set $K$ such that 
\begin{equation}
 y_{\alpha,\beta} = \int_{\mathbb{C}^n} z^\alpha \bar{z}^\beta d\mu ~, ~~~~ \text{for all} ~~ |\alpha|,|\beta| \leqslant d.
\end{equation}
In this section, we propose a solution to this problem. Precisely, we characterize when there exists a $\text{rank}M_d(y)$-atomic representing measure for the data $(y_{\alpha,\beta})_{|\alpha|,|\beta| \leqslant d}$.  We do so via the existence of an extension of the data, which must satisfy certain conditions, in the footsteps of Curto and Fialkow \cite{curto1991,curto-1996,curto-2000}. Such is how they characterize \cite[Theorem 5.1]{curto-2005} atomic measures $\mu$ supported on the semi-algebraic set $K$ such that 
\begin{equation}
 y_{\alpha,\beta} = \int_{\mathbb{C}^n} z^\alpha \bar{z}^\beta d\mu ~, ~~~~ \text{for all} ~~ |\alpha|+|\beta| \leqslant 2d
\end{equation}
given some complex numbers $(y_{\alpha,\beta})_{|\alpha|+|\beta| \leqslant 2d}$.
However, this moment problem is not relevant for the complex Lasserre hierarchy since the truncation of the data is different.
Below, we've represented the second order truncation in the complex hierarchy in blue and the second order truncation of Curto and Fialkow in black and blue:
$$
\begin{array}{cccccccc}
 & 1\hphantom{^2} & z\hphantom{^2} & \bar{z}^2 & \bar{z}^3 & \bar{z}^4 \\[.75em]
1\hphantom{^2} & \textcolor{blue}{\bullet} & \textcolor{blue}{\bullet} & \textcolor{blue}{\bullet} & \bullet & \bullet & \hdots \\[.75em]
z\hphantom{^2} & \textcolor{blue}{\bullet} & \textcolor{blue}{\bullet} & \textcolor{blue}{\bullet} & \bullet & \\[.75em]
z^2 &  \textcolor{blue}{\bullet} & \textcolor{blue}{\bullet} & \textcolor{blue}{\bullet} &  & \\[.25em]
z^3 &  \bullet & \bullet & & \ddots & \\[.75em]
z^4 &  \bullet &   &   &   & \\[.75em]
 & \vdots
\end{array}
$$
This leads to different notions of moment matrices. Below, we've represented the moment matrix of the complex hierarchy (on the left, in blue) and the moment matrix of Curto and Fialkow (on the right, in black):
$$
\begin{array}{cccc}
 & 1\hphantom{^2} & \bar{z}\hphantom{^2} & \bar{z}^2 \\[.75em]
1\hphantom{^2} & \textcolor{blue}{\bullet} & \textcolor{blue}{\bullet} & \textcolor{blue}{\bullet} \\[.75em]
z\hphantom{^2} & \textcolor{blue}{\bullet} & \textcolor{blue}{\bullet} & \textcolor{blue}{\bullet} \\[.75em]
z^2 &  \textcolor{blue}{\bullet} & \textcolor{blue}{\bullet} & \textcolor{blue}{\bullet}
\end{array}
~~~~~~~~~~~~~ 
\begin{array}{ccccccc}
 & 1\hphantom{^2} & \bar{z}\hphantom{^2} & z\hphantom{^2} & \bar{z}^2 & \bar{z}z & z^2 \\[.75em]
1\hphantom{^2} & \bullet & \bullet & \bullet & \bullet & \bullet & \bullet \\[.75em]
z\hphantom{^2} & \bullet & \bullet & \bullet & \bullet & \bullet & \bullet \\[.75em]
\bar{z}\hphantom{^2} & \bullet & \bullet & \bullet & \bullet & \bullet & \bullet \\[.75em]
z^2 & \bullet & \bullet & \bullet & \bullet & \bullet & \bullet \\[.75em]
z\bar{z} & \bullet & \bullet & \bullet & \bullet & \bullet & \bullet \\[.75em]
\bar{z}^2 & \bullet & \bullet & \bullet & \bullet & \bullet & \bullet
\end{array}
$$
The moment matrix in the complex hierarchy is referred to as \textit{pruned} complex moment matrix in \cite{lasserre2007}. However, the associated moment problem is not considered. Despite the discrepancies between the moment matrices, like Curto and Fialkow, we will rely on the notion of \textit{flat extension}, which is an extension of the moment matrix that preserves the positive semidefiniteness and the rank. 


Note that the complex moment problem of Curto and Fialkow is equivalent in some sense (see \cite[Theorem 5.2]{curto-2005}) to the real moment problem, i.e. where we seek a measure on a real semi-algebraic set such that
\begin{equation}
 y_{\alpha} = \int_{\mathbb{R}^{2n}} x^\alpha d\mu ~, ~~~~ \text{for all} ~~ |\alpha| \leqslant 2d
\end{equation}
given some real numbers $(y_{\alpha})_{|\alpha| \leqslant 2d}$. In contrast, the truncated moment problem arising in the complex hierarchy captures the real truncated moment problem as a special case. It corresponds to the case where the moment data forms a Hankel matrix (see Theorem \ref{th:special} below).

To provide a solution to the truncated moment problem arising in the complex hierarchy (Theorem \ref{th:hypo} below), we rely on the notion of hyponormality in operator theory. Indeed, we are unable to adapt the algebraic arguments used by Curto and Fialkow. They consider the ideal generated by the monomials that are indexes of the rows of their moment matrix. We are unable to make use of it in our context. 
We thus pursue a different approach, based exclusively on operator theory. The relationship between this discipline and the moment problem was recognized early on, as described by Akhiezer \cite[Chapter 4] {akhiezer-1965} in 1965. It has since been enriched by a vast literature including the works of Cassier \cite{cassier1984}, Schm\"{u}dgen \cite{schmudgen-1991}, and Putinar \cite{putinar-1993}. In particular, Atzmon \cite{atzmon-1975} used operator theory to solve the full moment problem on the unit disc in the complex plane. Later, Curto and Putinar \cite[Theorem 3.1]{curto1993} extended this result to subalgebraic subsets of the complex plane defined by one inequality. For such sets, the real moment problem was shown to be reducible to a complex moment problem in \cite{putinar1992}. In fact, Putinar employed this complexification of the real moment problem in his seminal result \cite{putinar-1993} which implies convergence of the (real) Lasserre hierarchy. Going back to the full complex moment problem, a solution was given in \cite{stochel1998} when the support is the entire complex plane. An operator-valued moment problem is also considered in that work. In the more recent paper \cite{kimsey2013}, a Hermitian-matrix-valued truncated moment problem is investigated in the multivariate setting. 

We now outline our approach. We need a few notations: let $\mathcal{B}(\mathcal{H})$ denote the set of linear bounded operators acting on a Hilbert space $(\mathcal{H}, \langle \cdot, \cdot \rangle)$. For all $T \in \mathcal{B}(\mathcal{H})$, let $T \succcurlyeq 0$ denote $\langle T u , u \rangle \geqslant 0$ for all $u \in \mathcal{H}$. In addition, the commutator of $A,B \in \mathcal{B}(\mathcal{H})$ is defined as $[A,B]:= AB - BA$. Finally, let $A^*$ denote the adjoint of $A \in \mathcal{B}(\mathcal{H})$. Following Halmos \cite{halmos1950}, an operator $T \in \mathcal{B}(\mathcal{H})$ is said to be $\hdots$
\begin{itemize}
\item \textit{normal} if $[T^*,T] = T^*T - TT^* = 0$;
\item \textit{subnormal} if it can be extended to a normal operator $N$ on a larger Hilbert space $\mathcal{K}$;
\item \textit{hyponormal} if $[T^*,T] = T^*T - TT^* \succcurlyeq 0$.
\end{itemize}
The notions of subnormality and hyponormality were introduced by Halmos in 1950 in order to extend the spectral theory of normal operators to a larger class of operators. They have since been used to shed light on the moment problem, as in \cite{athavale1990,vasilescu-2009} and the works cited above. The following implications hold (for explanations, see, e.g., \cite{curto-2010}):
$$\text{normal} ~~~ \Longrightarrow ~~~ \text{subnormal} ~~~ \Longrightarrow ~~~ \text{hyponormal} $$
The gap between subnormality and hyponormality has been the subject of much investigation, such as in \cite{curto1988,mccullough1989}. It was later discovered in \cite{curto1993} that there is in fact a significant gap: even  polynomially hyponormal operators (i.e. such that $p(T)$ is hyponormal for all $p\in \mathbb{C}[z]$) are not necessarily subnormal.
The key ingredient for our proof is that in finite dimension, normality, subnormality, and hyponormality are all equivalent. Indeed, if $\mathcal{H}$ is finite dimensional, then the trace of $[T^*,T]$ is equal to zero. If in addition $[T^*,T]\succcurlyeq 0$, then it must be that $[T^*,T] =0$. We next show how this observation is relevant for a tuple of operators. 

Following the definition of Athavale \cite{athavale1988}, operators $T_1, \hdots, T_n \in \mathcal{B}(\mathcal{H})$ are \textit{jointly hyponormal} if 
\begin{equation}
\label{eq:hypo}
\begin{pmatrix}
[T_1^*,T_1] & [T_2^*,T_1] & \hdots & [T_n^*,T_1] \\
[T_1^*,T_2] & [T_2^*,T_2] & \hdots & [T_n^*,T_2] \\
\vdots & \vdots & & \vdots \\
[T_1^*,T_n] & [T_2^*,T_n] & \hdots & [T_n^*,T_n]
\end{pmatrix} \succcurlyeq 0
\end{equation}
in the sense that for all $u_1, \hdots u_n \in \mathcal{H}$, there holds $\sum\limits_{i,j=1}^n \langle u_i, [T_j^*,T_i] u_j \rangle \geqslant 0$.\\
Thanks to our previous observation, in finite dimension, this is equivalent to\footnote{Make no confusion: this is not the definition of a normal tuple of operators, which is that $[T_i,T_j] = [T_i^*,T_i] = 0$ for all $i,j = 1, \hdots ,n$ (see \cite[p. 1505]{kimsey2016}).}
\begin{equation}
\begin{pmatrix}
[T_1^*,T_1] & [T_2^*,T_1] & \hdots & [T_n^*,T_1] \\
[T_1^*,T_2] & [T_2^*,T_2] & \hdots & [T_n^*,T_2] \\
\vdots & \vdots & & \vdots \\
[T_1^*,T_n] & [T_2^*,T_n] & \hdots & [T_n^*,T_n]
\end{pmatrix} = 0.
\end{equation}
This is itself trivially equivalent to 
\begin{equation}
\begin{pmatrix}
[T_i^*,T_i] & [T_j^*,T_i] \\
[T_i^*,T_j] & [T_j^*,T_j] 
\end{pmatrix} = 0, ~~~ \forall 1 \leqslant i < j \leqslant n,
\end{equation}
and to
\begin{equation}
\begin{pmatrix}
I & T_i^* & T_j^* \\
T_i & T_i^*T_i & T_i^*T_j  \\
T_j & T_i^*T_j & T_j^*T_j
\end{pmatrix} \succcurlyeq 0, ~~~ \forall 1 \leqslant i < j \leqslant n,
\end{equation}
thanks to a Schur complement.

The purpose of the third condition of Theorem \ref{th:hypo} below is to guarantee that joint hyponormality holds for a certain set of operators acting on a finite dimensional space. The purpose of the second condition of Theorem \ref{th:hypo} is to ensure that these operators commute (recall that $d_K := \max \{2,k_1,\hdots,k_m\}$). The operators in question are \textit{shift} operators, which are commonly used when dealing with the moment problem (e.g., \cite[Proposition 8]{putinar2008}). More explanations can be found in the proof.

\begin{theorem}
\label{th:hypo}
\normalfont
\textit{Consider a positive integer} $d$ \textit{and some complex numbers} $(y_{\alpha,\beta})_{|\alpha|,|\beta| \leqslant d}$\textit{. Assume that} $K$ \textit{contains a ball constraint} $|z_1|^2 + \hdots + |z_{n}|^2 \leqslant R^2$ \textit{for some radius} $R > 0$ \textit{and that we are in the multivariate setting} ($n>1$).
\textit{Then there exists a positive} $ \text{rank} M_{d}(y)$\textit{-atomic measure} $\mu$ \textit{supported on} $K$ \textit{such that}
\begin{equation}
 y_{\alpha,\beta} = \int_{\mathbb{C}^n} z^\alpha \bar{z}^\beta d\mu ~, ~~~~ \text{for all} ~~ |\alpha|,|\beta| \leqslant d
\end{equation}
\textit{if and only if there exists an extension} $(y_{\alpha,\beta})_{d < |\alpha|,|\beta| \leqslant d+d_K}$ \textit{such that:}
$$
\boxed{
\begin{array}{l}
\text{1.}~\textit{Positivity of moment and localizing matrices:}\\[.07cm] 
M_{d+d_K}(y) \succcurlyeq 0$ and $M_{d+d_K-k_i}(g_iy) \succcurlyeq 0, ~ i = 1, \hdots, m\\[.3cm]
\text{2.}~\textit{Commutativity of the shifts:}\\[.07cm]
\text{rank} M_{d+d_K}(y) = \text{rank} M_d(y)\\[.3cm]
\text{3.}~\textit{Joint hyponormality of the shifts:}\\[.1cm] 
\begin{pmatrix} 
M_{d}(y) & M_{d}(z_i y) & M_{d}(z_j y) \\ 
M_{d}(\bar{z}_i y) & M_{d}(|z_i|^2 y)  & M_{d}(z_j \bar{z}_i y) \\
M_{d}(\bar{z}_j y) & M_{d}(z_i \bar{z}_j y) & M_{d}(|z_j|^2 y)
\end{pmatrix} 
\succcurlyeq 0 , ~ \forall 1 \leqslant i < j \leqslant n.
\end{array}
}
$$
\end{theorem}
\begin{proof}
See Appendix \ref{app:hypo}.
\end{proof}
In the univariate setting ($n=1$), the ``joint hyponormality of the shifts'' condition must be replaced by
\begin{equation}
\begin{pmatrix} 
M_{d}(y) & M_{d}(z y) \\ 
M_{d}(\bar{z} y) & M_{d}(|z|^2 y)
\end{pmatrix}
\succcurlyeq 0.
\end{equation}
Theorem \ref{th:hypo} then holds with $d_K := \max \{1,k_1,\hdots,k_m\}$, in contrast to the multivariate setting where $d_K := \max \{2,k_1,\hdots,k_m\}$ (see proof for explanations).

In the next result, we consider two cases where the ball constraint and the ``joint hyponormality of the shifts'' condition can be removed. One case is when the moment data forms a Toeplitz matrix, that is to say when $y_{\alpha,\beta} = \overline{y_{\beta,\alpha}}$ only depends on $\alpha-\beta$. This is revelant when optimizing in the presence of the constraints $|z_k|^2 = 1,~ k = 1, \hdots,n$. The other case is when the moment data forms a Hankel matrix, that is to say when $y_{\alpha,\beta}$ is real and only depends on $\alpha+\beta$. This is relevant for real polynomial optimization, which can be viewed as an instance of complex polynomial optimization  with the constraints $\textbf{i}z_k - \textbf{i}\bar{z}_k = 0, ~ k = 1,\hdots,n$. It corresponds exactly to the moment data generated by the original (real) Lasserre hierarchy.

\begin{theorem}
\label{th:special}
\normalfont
\textit{Consider a positive integer} $d$ \textit{and some complex numbers} $(y_{\alpha,\beta})_{|\alpha|,|\beta| \leqslant d}$\textit{. Assume that} $K$ \textit{contains either the constraints} $|z_k|^2 = 1, ~ k = 1,\hdots,n$ \textit{or the constraints} $\textbf{i}z_k - \textbf{i}\bar{z}_k = 0, ~ k = 1,\hdots,n$.
\textit{Then there exists a positive} $ \text{rank} M_{d}(y)$\textit{-atomic measure} $\mu$ \textit{supported on} $K$ \textit{such that}
\begin{equation}
\label{eq:atomic_repbis}
 y_{\alpha,\beta} = \int_{\mathbb{C}^n} z^\alpha \bar{z}^\beta d\mu ~, ~~~~ \text{for all} ~~ |\alpha|,|\beta| \leqslant d
\end{equation}
\textit{if and only if there exists an extension} $(y_{\alpha,\beta})_{d < |\alpha|,|\beta| \leqslant d+d_K}$ \textit{such that:}\\[.1cm]
$
\begin{array}{l}
\text{1.}~\textit{Positivity of moment and localizing matrices:}\\
M_{d+d_K}(y) \succcurlyeq 0$ and $M_{d+d_K-k_i}(g_iy) \succcurlyeq 0, ~ i = 1, \hdots, m\\[.15cm]
\text{2.}~\textit{Commutativity of the shifts:}\\
\text{rank} M_{d+d_K}(y) = \text{rank} M_d(y).
\end{array}
$
\end{theorem}
\vspace{.1cm}
\begin{proof}
See Appendix \ref{app:special}.
\end{proof}

Our solution to the moment problem in the Toeplitz case is new to the best of our knowledge.
In the univariate case $n=1$ with support equal to the full space $K = \mathbb{C}$, it corresponds to the truncated trigonometric moment problem. A solution to this problem has been given by \cite[P. 211]{iohvidov1982}, \cite[Theorem I.I.12]{akhiezer1962}, and \cite[Theorem 6.12]{curto1991}. It can be stated as follows. A Toeplitz matrix can be represented by a positive Borel measure if and only it is positive semidefinite. In other words, there need not exist a flat extension for there to exist a measure. For some more recent work on the trigonometric moment problem, see \cite{gabardo1999,li2006,zag2015}. See also \cite{bakonyi2011} for its relevance in the context of matrix completions.

The Hankel case in Theorem \ref{th:special} corresponds to the solution of real truncated moment found in \cite[Theorem 3.11]{lasserre-2010} due to Curto and Fialkow~\cite[Theorem 1.1]{curto-2005}. However, their result is stronger because it only requires that $d_K \geqslant 1$, while we require that $d_K \geqslant 2$. The reason why we record this result is to underscore the link between the moment problems arising in the real and complex hierarchies. It also provides a new proof based solely on operator theory, in contrast to the proof of Curto and Fialkow, and the more recent proof of Laurent \cite{laurent2005}. The latter relies partly on algebraic tools, while the former relies only on algebraic tools. Note that the result has been generalized to moment matrices indexed by arbritary monomials in \cite{mourrain2009}.

\section{Multi-ordered Lasserre hierarchy}
\label{sec:Multi-ordered Lasserre hierarchy}

In \cite{mh_sparse_msdp}, a heuristic was proposed to exploit sparsity in the Lasserre hierarchy when applied to the optimal power flow problem. Inspired by that work, we propose a general approach to exploit sparsity in any polynomial optimization problem (in real or complex variables) which preserves global convergence.
The approach associates a different relaxation order to each constraint, in contrast to the Lasserre hierarchy which associates the same relaxation order to all constraints. 

\subsection{Defining a relaxation order at each constraint}
\label{subsec:Defining a relaxation order at each constraint}

In order to define a relaxation order for each constraint, we build on the work of Waki \textit{et al.} \cite{waki-2006}. Those authors propose to use chordal sparsity in the Lasserre hierarchy. They draw on the \textit{correlative sparsity} graph whose vertices are the variables and whose edges signify that two variables appear simultaneously either in the objective or a constraint. The idea of Waki \textit{et al.} is to restrain the variables appearing in the sum-of-squares (\textit{a priori} all variables) to subsets of variables. Indeed, in the sum-of-squares decomposition, i.e.
\begin{equation}
\label{eq:dense}
f - \lambda = \sigma_0 + \sum_{i=1}^m \sigma_i g_i
\end{equation}
one would like to restrain the variables appearing in $\sigma_i$ in function of the variables appearing in the constraint $g_i$. For instance, if the variables appearing in one constraint $g_i$ are $x_1,x_2,x_3,x_4$ (among say $x_1,\hdots,x_{100}$) one could hope to restrain the variables appearing in $\sigma_i$ to $x_1, x_2,x_3,x_4$ (or some slightly larger set). This hope becomes a reality when considering the maximal cliques of a chordal extension of the correlative sparsity graph. Then, to each constraint $g_i$, one can associate a maximal clique containing all the variables of $g_i$ (preferably with the fewest number of variables if several cliques work). Next, one can restrain the variables in the sum-of-squares $\sigma_i$ to that clique. The sum-of-squares $\sigma_0$ can be restricted to a sum of terms, where each term is a sum-of-squares with variables belonging to a clique.
At a given order, the relaxation might be weaker but global convergence is preserved, as was first shown by Lasserre \cite[Theorems 2.28 and 4.7]{lasserre-2010}, and later confirmed in \cite{schweighofer-2007} and~\cite{kuhlmann-2007}. These results easily generalize to complex numbers: the proof is the same as in the real case~\cite[Lemma B.13 and 4.10.2 Proof of Theorem 4.7]{lasserre-2010} once the real vector spaces on which measures are defined are replaced by complex vector spaces. Note that the assumption of a redundant ball constraint per clique must be replaced by a sphere and slack variable per clique. To sum up, if $C_1,\hdots,C_p$ are the cliques, Waki \textit{et al.} propose to restrain \eqref{eq:dense} to
\begin{equation}
\label{eq:sparse}
f - \lambda ~~~ =  ~~~ \sum_{k=1}^p ~~ \left( \sigma_0^k  ~~~~ +  \sum_{\tiny \begin{array}{c} \text{constraints}~ i \\ \text{associated to}~C_k \end{array}} \sigma_i g_i \right)
\end{equation}
where the variables of $\sigma_0^k$ and $\sigma_i$ are restrained to the clique $C_k$.

The approach of Waki \textit{et al.} reduces the computational burden of the Lasserre hierarchy for sparse problems. Concerning the optimal power flow problem, it allows one to solve some hard instances to global optimality with up to 80 variables \cite{ibm_paper} (instead of 20 without exploiting sparsity \cite{pscc2014}). However, by using the correlative sparsity graph discussed above, a lot of the sparsity is lost. We thus propose a finer notion of sparsity that takes advantage of the fact that the objective and constraints are polynomials. To that effect, we define the \textit{monomial sparsity} graph whose vertices are the variables and whose edges signify that two variables appear simultaneously in a monomial of either the objective or a constraint. We can then define a relaxation order $d_i$ for each constraint $g_i$. If we want a constraint $g_i$ to have a high order, i.e. with a sum-of-squares $\sigma_i$ of degree greater than zero, then we add the correlative sparsity induced by $g_i$ to the monomial sparsity graph. We then consider the maximal cliques of a chordal extension of the resulting graph. To each constraint $g_i$ that is of high order, one can associate a maximal clique containing all the variables of $g_i$. The variables in the sum-of-squares $\sigma_i$ can then be restrained to the clique associated to $g_i$ when it is of high order; if not of high order, the sum-of-squares is a nonnegative real number. The sum-of-squares $\sigma_0$ can be restricted to a sum of terms, where each term is a sum-of-squares with variables belonging to a clique. To sum up, we replace \eqref{eq:sparse} by
\begin{equation}
\label{eq:multi}
f - \lambda ~~~ =  ~~~ \sum_{k=1}^p ~~ \left( \sigma_0^k  ~~~~ +  \sum_{\tiny \begin{array}{c} \text{constraints}~ i \\ \text{of high order} \\ \text{associated to}~C_k \end{array}} \sigma_i g_i \right) + \sum_{\tiny \begin{array}{c} \text{constraints}~ i \\ \text{not of high order} \end{array}} \sigma_i g_i
\end{equation}
where the variables of $\sigma_0^k$ are restrained to the clique $C_k$, and same goes for $\sigma_i$ in the case of high order constraints associated to $C_k$; otherwise $\sigma_i$ is a nonnegative real number. The polynomial $\sigma_0^k$ is a sum of squares of polynomials of degree less than or equal to the maximal relaxation order $d_i$ among all high order constraints associated to $C_k$. If no high order constraints are associated to $C_k$, the degree is less than or equal to one. As can be seen in \eqref{eq:multi}, if all the constraints have a high order, we are back to \eqref{eq:sparse}, i.e. the approach of Waki \textit{et al.}\footnote{The monomial sparsity graph augmented with the correlative sparsity induced by each constraint is none other than the correlative sparsity graph, provided that the correlative sparsity induced by the objective is included in the correlative sparsity induced by the constraints.}, but with different orders at each constraint. Global convergence is thus preserved as the minimal order increases to infinity.
\begin{example}
\normalfont
\textit{Consider the following optimization problem:}
\begin{equation}
\inf_{x_1,x_2,x_3,x_4 \in \mathbb{R}} ~~~ x_1x_2 + x_1x_4 ~~~~~ \text{s.t.} ~~~~~ \left\{ \begin{array}{r}
 x_1x_2 + x_1x_3 \geqslant 0 \\
 x_1x_3 + x_1x_4 + x_1x_2 \geqslant 0
\end{array}
\right.
\end{equation}
\textit{It is solely meant to illustrate the above notions; the next example is much more interesting from a numerical perspective. Figure} \ref{fig:sparsity} \textit{illustrates the correlative sparsity used by Waki et al. and the monomial sparsity advocated in this paper. Suppose one wants to impose order 2 at the first constraint (i.e. a high order) and order 1 at the second constraint (i.e. not a high order).  The correlative sparsity induced by the first constraint is the triangle formed by the first three variables. When added to the monomial sparsity pattern, it yields the graph on the left of Figure} \ref{fig:sparsity} \textit{if the edge} (2,3) \textit{is added. The variables in the sum-of-squares} $\sigma_1$ \textit{can then be restrained to} $x_1,x_2,x_3$\textit{, while the sum-of-squares} $\sigma_2$ \textit{is a nonnegative real number. The reason why we must add the correlative sparsity induced by the high order constraint to the monomial sparsity pattern is because in the expression} $\sigma_1(x_1,x_2,x_3) g_1(x_1,x_2,x_3)$\textit{, all the possible products} $x_1x_2,x_1x_3,x_2x_3$ \textit{appear.}  
\begin{figure}[!h]
\centering
\includegraphics[width=8cm]{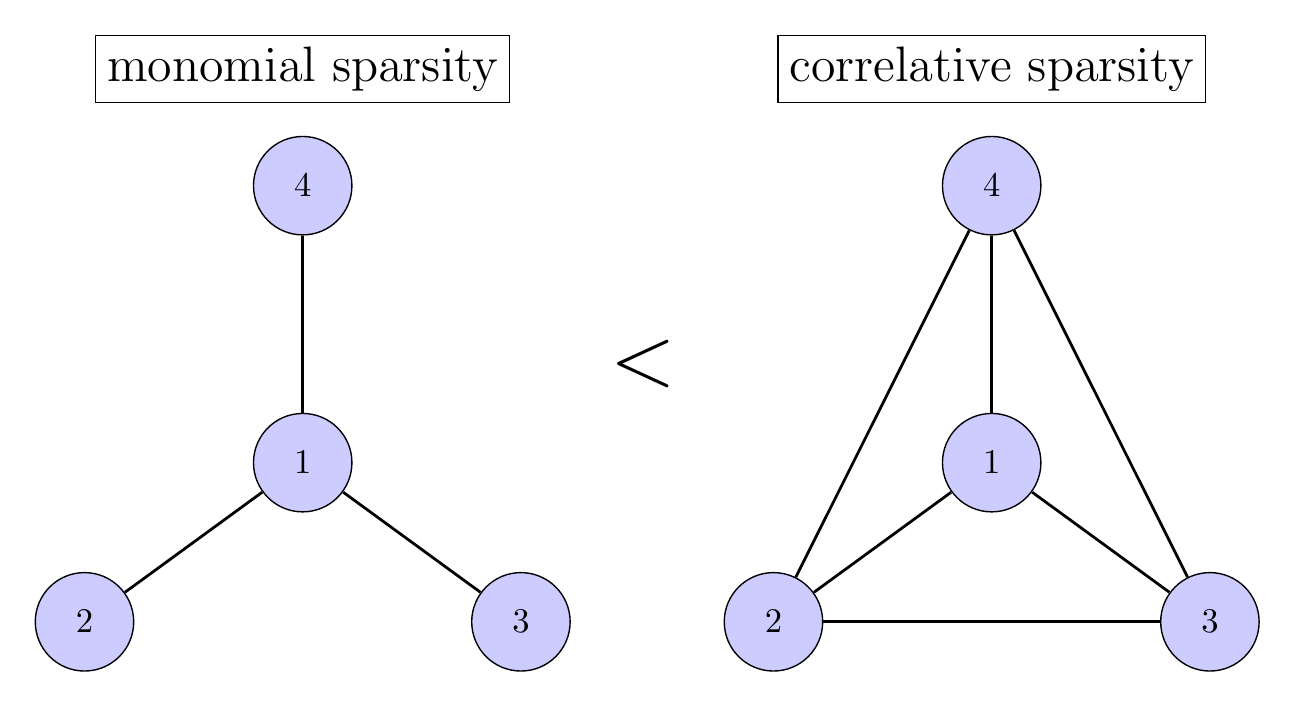}
\label{fig:sparsity}
\caption{Two different notions of sparsity}
\end{figure}
\end{example}

\begin{example}
\normalfont
\textit{In}~\cite[WB5, $Q_5^{\text{min}} = -30.00$ MVAr]{bukhsh2013} \textit{an instance of the optimal power flow problem is proposed. It can be viewed as a complex polynomial optimization problem with five variables} $z_1,z_2,z_3,z_4,z_5 \in \mathbb{C}$\textit{. Let} $\text{mon}(\cdot)$ \textit{denote the monomial sparsity induced either by the objective} $f$ \textit{or by one of the constraints} $g_{1},\hdots, g_{20}$\textit{:}
\begin{align*}
\text{mon}(f) =\; & \{ (1,2),(1,3),(3,5),(4,5) \} \\
\text{mon}(g_{1}) = \text{mon}(g_{2}) =\; & \{(1,2),(1,3)\} \\
\text{mon}(g_{3}) = \text{mon}(g_{4}) =\; & \{ (1,2),(2,3),(2,4) \} \\
\text{mon}(g_{5}) = \text{mon}(g_{6}) =\; & \{ (1,3),(2,3),(3,5) \} \\
\text{mon}(g_{7}) = \text{mon}(g_{8}) =\; & \{ (2,4),(4,5) \}  \\ \addtocounter{equation}{1}\tag{\theequation}\label{eq:5bus}
\text{mon}(g_{9}) = \text{mon}(g_{10}) =\; & \{ (3,5),(4,5) \} \\
\text{mon}(g_{11}) = \hdots = \text{mon}(g_{20}) =\; & \emptyset  
\end{align*}
\noindent \textit{The monomial sparsity is empty when no two distinct variables appear in one monomial, such as in} $g_{11}(z,\bar{z}) = z_1 \bar{z}_1 - 0.90$\textit{. Otherwise, it corresponds to all the couples of distinct variables that appear in one monomial, such as} $(z_1,z_2)$ and $(z_1,z_3)$ \textit{in} $g_1(z,\bar{z}) = 8.12 z_1 \bar{z}_1 \!-\!(2.06\!-\!4.64 \mathbf{i}) z_2\bar{z}_1\!-\!(2.06\!+\!4.64\mathbf{i}) z_1 \bar{z}_2\!-\!(2.00\!-\!4.00\mathbf{i}) z_3 \bar{z}_1\!-\!(2.00\!+\!4.00\mathbf{i}) z_1 \bar{z}_3$.
\textit{The complex hierarchy with} $d_i = 1,\; \forall i \in \left\lbrace 1,2,3,4,5,6,11,12,13,14,15,16 \right\rbrace$, and $d_i = 2,\; \forall i \in \left\lbrace 7,8,9,10,17,18,19,20 \right\rbrace$\textit{, yields a global solution.} (\textit{Second-order constraints are identified using the procedure described in the next section.}) \textit{With this choice of high order constraints, the relevant graph is illustrated in Figure} \ref{fig:5bus graph}\textit{. It is already chordal and its maximal cliques are} \{1,2,3\} \textit{and} \{2,3,4,5\}\textit{. We can associate the latter to all high order constraints since it contains their variables. The globally optimal objective value thus obtained is} 946.6 MW (\textit{in accordance with} \cite{bukhsh-archive}) \textit{with corresponding decision variable} $z = (1.0467 + 0.0000\mathbf{i}, 0.9550 - 0.0578\mathbf{i}, 0.9485 - 0.0533\mathbf{i}, 0.7791 + 0.6011\mathbf{i}, 0.7362 + 0.7487\mathbf{i})^T$.
\begin{figure}[!h]
\centering
\includegraphics[width=4cm]{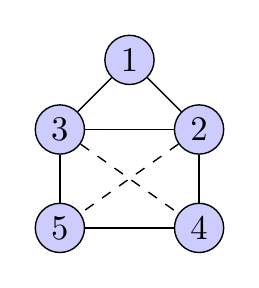}
\label{fig:5bus graph}
\caption{Monomial sparsity graph (solid lines) plus correlative sparsity of high order constraints (dashed lines)}
\end{figure}
%
%

\end{example}

\subsection{Updating the relaxation order at each constraint}
\label{subsec:Updating the relaxation order at each constraint}
Consider a polynomial optimization problem with 10,000 constraints, as encountered in Table \ref{tab:results}. Say that one has computed the first order Lasserre relaxation and that it does not yield a global solution. How does one choose the constraints at which to augment the relaxation order? Even if only 2 constraints require a high order, the combinatorial difficulty is tremendous: there are 49,995,000 combinations to choose from. This section provides one way to choose the high order constraints (there could be other ways of course). We next present the approach when applied to polynomial optimization in real numbers (i.e. $\inf_{x\in \mathbb{R}^n} \sum_{\alpha} f_{\alpha} x^{\alpha} ~\text{s.t.}~ \sum_{\alpha} g_{i,\alpha} x^{\alpha} \geqslant 0, ~ i = 1,\hdots,m$), but it also applies to complex numbers. We begin by computing a solution $y$ to the moment relaxation with the lowest possible order at each constraint. Next, we do the following:\\\\
Until a measure can be extracted from a solution $y$ to the moment relaxation:\\
\begin{enumerate}
\item find a closest measure $\mu$ to $y$ not necessarily supported on $K$:
\begin{equation}
\label{eq:step2}
\underset{\mu}{\arg\min} ~~ \sum_{\alpha} \left( y_{\alpha} - \int_{\mathbb{R}^n} x^{\alpha} d\mu \right)^2
\end{equation}
\item increment $d_i = d_i+1$ at the largest mismatch, that is to say:
\begin{equation}
\underset{1\leqslant i \leqslant m}{\arg\max} ~ \left| \sum_{\alpha} g_{i,\alpha} \left( y_{\alpha} - \int_{\mathbb{R}^n} x^\alpha d\mu \right) \right|
\end{equation}
\item compute a solution $y$ to the moment relaxation of order $(d_1,\hdots,d_m)$.\\
\end{enumerate}

The first step is \textit{a priori} challenging computationally, so we use a proxy for it. For each clique $C_k$, we have a set of pseudo-moments $(y_{\alpha+\beta})_{|\alpha|,|\beta|=1}$ from which we can extract an eigenvector of highest eigenvalue. This eigenvector $u_k$ is defined up to a sign change when dealing with real numbers since $(-u_k)(-u_k)^T = u_k u_k^T$ (respectively up to phase shift when dealing with complex numbers since $(e^{\textbf{i}\theta}u_k)(e^{\textbf{i}\theta}u_k)^* = u_k u_k^*$). In order to synchronize the eigenvectors among the overlapping cliques, we must therefore choose the signs (respectively phase shifts), which can be done approximately via convex optimization. Next, we use a least-squares optimization to find a vector $x \in \mathbb{R}^n$ (respectively $z \in \mathbb{C}^n$) that best matches the signed eigenvectors on each clique (respectively phased eigenvectors). This approximately provides a closest measure to the pseudo-moments, namely the Dirac measure with atom equal to $x$ (respectively $z$) and weight equal to 1.

The second step depends on three parameters: a mismatch tolerance $\epsilon > 0$; the number $h$ of largest mismatches considered at each iteration; and an upper bound $\Delta^{\text{max}}_{\text{min}}$ on the difference between maximum and minimum relaxation orders. (In the experiments of Table \ref{tab:results}, we take $\epsilon = 1$ MVA, $h = 2$, and $\Delta^{\text{max}}_{\text{min}} = 2$.) First, assume that there are constraints with a mismatch greater than $\epsilon$, i.e. $|L_y(g_i) - g_i(x)| > \epsilon$ (respectively $|L_y(g_i) - g_i(z,\bar{z})| > \epsilon$), and whose relaxation orders have not yet been increased. Then increment the order at those that have the $h$ largest mismatches and consider a set of cliques which contain all their variables; increment also the order of any constraint whose variables are included in those cliques. For all other constraints, keep the same relaxation order unless the bound $\Delta^{\text{max}}_{\text{min}}$ is violated, in which case increment all those with the smallest order. Second, assume that all the constraints with a mismatch greater than $\epsilon$ have already been augmented. Then, among those, consider the $h$ largest mismatches, and repeat the above procedure. Third, if all the mismatches are below $\epsilon$, then the point $x$ (respectively $z$) is feasible up to $\epsilon$ and globally optimal.

\section{Block diagonal Lasserre hierarchy} 
\label{sec:Block diagonal Lasserre hierarchy} 

Finally, we exhibit a block diagonal structure of the Lasserre hierarchy in the presence of symmetries. We begin by an illustrative example.

\begin{example}
\normalfont
\textit{In} \cite[WB2, $V_2^{\text{max}} = 1.022 ~\text{p.u.}$]{bukhsh2013}\textit{, an instance of the optimal power flow is proposed. It yields the following complex polynomial optimization problem}
$$
\begin{array}{c}
\inf\limits_{z_1,z_2\in \mathbb{C}} ~~ 8  |z_1-z_2|^2 \\[.5cm]
  \text{s.t.} ~~~
\left\{
\begin{array}{rcl}
0.9025 \leqslant |z_1|^2 \leqslant 1.1025 ~~~~~~   & & \\[.5em]
0.9025 \leqslant |z_2|^2 \leqslant 1.0568 ~~~~~~ & & \\[.5em]
( 2 + 10\textbf{i}) z_1\bar{z}_2 ~+~ \hphantom{+}( 2 - 10\textbf{i})z_2 \bar{z}_1 ~-~ \hphantom{0}4 |z_2|^2 & = & \hphantom{-}350 \\[.5em]
 ( -10 + 2\textbf{i}) z_1 \bar{z}_2 ~+~ ( -10 - 2\textbf{i}) z_2 \bar{z}_1 ~+~ 20 |z_2|^2 & = & -350
\end{array}
\right.
\end{array}
$$
\textit{Notice that if} $(z_1,z_2)$ \textit{is a feasible point, then so is} $(e^{\textbf{i}\theta}z_1,e^{\textbf{i}\theta}z_2)$ \textit{for all} $\theta \in \mathbb{R}$.
\textit{When converted to real numbers} $z_1:=x_1 +x_3 \textbf{i}$ \textit{and} $z_2:= x_2+x_4\textbf{i}$\textit{, it yields}
$$
\begin{array}{c}
\inf\limits_{x_1,x_2,x_3,x_4 \in \mathbb{R}} ~~~ 8  (x_1-x_2)^2 + 8  (x_3-x_4)^2\\[.5cm]
  \text{s.t.} ~~~
\left\{
\begin{array}{rcl}
0.9025 \leqslant x_1^2 + x_3^2 \leqslant 1.1025 ~~~~~~   & & \\[.5em]
0.9025 \leqslant x_2^2 + x_4^2 \leqslant 1.0568 ~~~~~~ & & \\[.5em]
\hphantom{-0}4x_1x_2 + \hphantom{0}4x_3x_4 + 20 x_1x_4 - 20 x_3x_2 - \hphantom{0}4 x_2^2 + \hphantom{0}4 x_4^2 & = & \hphantom{-}350 \\[.5em]
-20 x_1x_2 -20 x_3x_4 + \hphantom{0}4 x_1x_4 - \hphantom{0}4 x_3x_2 + 20 x_2^2 + 20 x_4^2 & = & -350
\end{array}
\right.
\end{array}
$$
\textit{Notice that if} $(x_1,x_2,x_3,x_4)$ \textit{is a feasible point, then so is} $(-x_1,-x_2,-x_3,-x_4)$.

\textit{The above symmetries allow one to cancel many terms in the Lasserre hierarchy at no loss of bound quality. We next illustrate this. The real and complex hierarchies yield the same bounds at the first, second, and third orders} (888.1, 894.3, and 905.7 MW \textit{respectively})\textit{. This is in accordance with} \cite[Table I]{cedric_tps}\textit{. The rank of the real and complex moment matrices guarantee that global convergence is reached at the third order. At that order, one can set to zero the following terms in the moment matrices (the bullets represent potentially non-zero terms):}\\

\textit{Complex moment matrix:}
$$
\tiny
\arraycolsep=.1cm
\begin{array}{rcccccccccc}
&  \begin{turn}{90}$1$\end{turn} & \begin{turn}{90}$\bar{z}_1$\end{turn} & \begin{turn}{90}$\bar{z}_2$\end{turn} & \begin{turn}{90}$\bar{z}_1 \bar{z}_1$\end{turn} & \begin{turn}{90}$\bar{z}_1\bar{z}_2$\end{turn} & \begin{turn}{90}$\bar{z}_2\bar{z}_2$\end{turn} & \begin{turn}{90}$\bar{z}_1\bar{z}_1\bar{z}_1$\end{turn} & \begin{turn}{90}$\bar{z}_1\bar{z}_1\bar{z}_2$\end{turn} & \begin{turn}{90}$\bar{z}_1\bar{z}_2\bar{z}_2$\end{turn} & \begin{turn}{90}$\bar{z}_2\bar{z}_2\bar{z}_2$\end{turn} \\[.25em]
1 & \bullet & 0 & 0 & 0 & 0 & 0 & 0 & 0 & 0 & 0 \\[.25em]
z_1 & 0 & \bullet & \bullet & 0 & 0 & 0 & 0 & 0 & 0 & 0 \\[.25em]
z_2 & 0 & \bullet & \bullet & 0 & 0 & 0 & 0 & 0 & 0 & 0 \\[.25em]
z_1 z_1 & 0 & 0 & 0 & \bullet & \bullet & \bullet & 0 & 0 & 0 & 0 \\[.25em]
z_1 z_2 & 0 & 0 & 0 & \bullet & \bullet & \bullet & 0 & 0 & 0 & 0 \\[.25em]
z_2 z_2 & 0 & 0 & 0 & \bullet & \bullet & \bullet & 0 & 0 & 0 & 0 \\[.25em]
z_1z_1z_1 & 0 & 0 & 0 & 0 & 0 & 0 & \bullet & \bullet & \bullet & \bullet \\[.25em]
z_1z_1z_2 & 0 & 0 & 0 & 0 & 0 & 0 & \bullet & \bullet & \bullet & \bullet \\[.25em]
z_1z_2 z_2 & 0 & 0 & 0 & 0 & 0 & 0 & \bullet & \bullet & \bullet & \bullet \\[.25em]
z_2z_2z_2 & 0 & 0 & 0 & 0 & 0 & 0 & \bullet & \bullet & \bullet & \bullet \\[.25em]
\end{array}
$$

\textit{Real moment matrix:}
$$
\tiny
\arraycolsep=.1cm
\begin{array}{rccccccccccccccccccccccccccccccccccc}
 & \begin{turn}{90}$1$\end{turn} & \begin{turn}{90}$x_1$\end{turn} & \begin{turn}{90}$x_2$\end{turn} & \begin{turn}{90}$x_3$\end{turn} & \begin{turn}{90}$x_4$\end{turn} & \begin{turn}{90}$x_1x_1$\end{turn} & \begin{turn}{90}$x_1x_2$\end{turn} & \begin{turn}{90}$x_1 x_3$\end{turn} & \begin{turn}{90}$x_1 x_4$\end{turn} & \begin{turn}{90}$x_2x_2$\end{turn} & \begin{turn}{90}$x_2x_3$\end{turn} & \begin{turn}{90}$x_2x_4$\end{turn} & \begin{turn}{90}$x_3x_3$\end{turn} & \begin{turn}{90}$x_3x_4$\end{turn} & \begin{turn}{90}$x_4x_4$\end{turn} & \begin{turn}{90}$x_1x_1x_1$\end{turn} & \begin{turn}{90}$x_1x_1x_2$\end{turn} & \begin{turn}{90}$x_1x_1x_3$\end{turn} & \begin{turn}{90}$x_1x_1x_4$\end{turn} &\begin{turn}{90}$x_1x_2x_2$\end{turn} & \begin{turn}{90}$x_1x_2x_3$\end{turn} & \begin{turn}{90}$x_1x_2x_4$\end{turn} & \begin{turn}{90}$x_1x_3x_3$\end{turn} & \begin{turn}{90}$x_1x_3x_4$\end{turn} & \begin{turn}{90}$x_1x_4x_4$\end{turn} & \begin{turn}{90}$x_2x_2x_2$\end{turn} & \begin{turn}{90}$x_2x_2x_3$\end{turn} & \begin{turn}{90}$x_2x_2x_4$\end{turn} & \begin{turn}{90}$x_2x_3x_3$\end{turn} &\begin{turn}{90}$x_2x_3x_4$\end{turn} & \begin{turn}{90}$x_2x_4x_4$\end{turn} & \begin{turn}{90}$x_3x_3x_3$\end{turn} & \begin{turn}{90}$x_3x_3x_4$\end{turn} & \begin{turn}{90}$x_3x_4x_4$\end{turn} & \begin{turn}{90}$x_4x_4x_4$\end{turn} \\[.25em]
1 & \bullet & 0 & 0 & 0 & 0 & \bullet & \bullet & \bullet & \bullet & \bullet & \bullet & \bullet & \bullet & \bullet & \bullet & 0 & 0 & 0 & 0 & 0 & 0 & 0 & 0 & 0 & 0 & 0 & 0 & 0 & 0 & 0 & 0 & 0 & 0 & 0 & 0 \\[.25em]
x_1 & 0 & \bullet & \bullet & \bullet & \bullet & 0 & 0 & 0 & 0 & 0 & 0 & 0 & 0 & 0 & 0 & \bullet & \bullet & \bullet & \bullet & \bullet & \bullet & \bullet & \bullet & \bullet & \bullet & \bullet & \bullet & \bullet & \bullet & \bullet & \bullet & \bullet & \bullet & \bullet & \bullet \\[.25em]
x_2 & 0 & \bullet & \bullet & \bullet & \bullet & 0 & 0 & 0 & 0 & 0 & 0 & 0 & 0 & 0 & 0 & \bullet & \bullet & \bullet & \bullet & \bullet & \bullet & \bullet & \bullet & \bullet & \bullet & \bullet & \bullet & \bullet & \bullet & \bullet & \bullet & \bullet & \bullet & \bullet & \bullet \\[.25em]
x_3 & 0 & \bullet & \bullet & \bullet & \bullet & 0 & 0 & 0 & 0 & 0 & 0 & 0 & 0 & 0 & 0 & \bullet & \bullet & \bullet & \bullet & \bullet & \bullet & \bullet & \bullet & \bullet & \bullet & \bullet & \bullet & \bullet & \bullet & \bullet & \bullet & \bullet & \bullet & \bullet & \bullet \\[.25em]
x_4 & 0 & \bullet & \bullet & \bullet & \bullet & 0 & 0 & 0 & 0 & 0 & 0 & 0 & 0 & 0 & 0 & \bullet & \bullet & \bullet & \bullet & \bullet & \bullet & \bullet & \bullet & \bullet & \bullet & \bullet & \bullet & \bullet & \bullet & \bullet & \bullet & \bullet & \bullet & \bullet & \bullet \\[.25em]
x_1x_1 & \bullet & 0 & 0 & 0 & 0 & \bullet & \bullet & \bullet & \bullet & \bullet & \bullet & \bullet & \bullet & \bullet & \bullet & 0 & 0 & 0 & 0 & 0 & 0 & 0 & 0 & 0 & 0 & 0 & 0 & 0 & 0 & 0 & 0 & 0 & 0 & 0 & 0 \\[.25em]
x_1x_2& \bullet & 0 & 0 & 0 & 0 & \bullet & \bullet & \bullet & \bullet & \bullet & \bullet & \bullet & \bullet & \bullet & \bullet & 0 & 0 & 0 & 0 & 0 & 0 & 0 & 0 & 0 & 0 & 0 & 0 & 0 & 0 & 0 & 0 & 0 & 0 & 0 & 0 \\[.25em]
x_1x_3&  \bullet & 0 & 0 & 0 & 0 & \bullet & \bullet & \bullet & \bullet & \bullet & \bullet & \bullet & \bullet & \bullet & \bullet & 0 & 0 & 0 & 0 & 0 & 0 & 0 & 0 & 0 & 0 & 0 & 0 & 0 & 0 & 0 & 0 & 0 & 0 & 0 & 0 \\[.25em]
x_1x_4 & \bullet & 0 & 0 & 0 & 0 & \bullet & \bullet & \bullet & \bullet & \bullet & \bullet & \bullet & \bullet & \bullet & \bullet & 0 & 0 & 0 & 0 & 0 & 0 & 0 & 0 & 0 & 0 & 0 & 0 & 0 & 0 & 0 & 0 & 0 & 0 & 0 & 0 \\[.25em]
x_2x_2 & \bullet & 0 & 0 & 0 & 0 & \bullet & \bullet & \bullet & \bullet & \bullet & \bullet & \bullet & \bullet & \bullet & \bullet & 0 & 0 & 0 & 0 & 0 & 0 & 0 & 0 & 0 & 0 & 0 & 0 & 0 & 0 & 0 & 0 & 0 & 0 & 0 & 0 \\[.25em]
x_2x_3 & \bullet & 0 & 0 & 0 & 0 & \bullet & \bullet & \bullet & \bullet & \bullet & \bullet & \bullet & \bullet & \bullet & \bullet & 0 & 0 & 0 & 0 & 0 & 0 & 0 & 0 & 0 & 0 & 0 & 0 & 0 & 0 & 0 & 0 & 0 & 0 & 0 & 0 \\[.25em]
x_2x_4 & \bullet & 0 & 0 & 0 & 0 & \bullet & \bullet & \bullet & \bullet & \bullet & \bullet & \bullet & \bullet & \bullet & \bullet & 0 & 0 & 0 & 0 & 0 & 0 & 0 & 0 & 0 & 0 & 0 & 0 & 0 & 0 & 0 & 0 & 0 & 0 & 0 & 0 \\[.25em]
x_3x_3& \bullet & 0 & 0 & 0 & 0 & \bullet & \bullet & \bullet & \bullet & \bullet & \bullet & \bullet & \bullet & \bullet & \bullet & 0 & 0 & 0 & 0 & 0 & 0 & 0 & 0 & 0 & 0 & 0 & 0 & 0 & 0 & 0 & 0 & 0 & 0 & 0 & 0 \\[.25em]
x_3x_4 & \bullet & 0 & 0 & 0 & 0 & \bullet & \bullet & \bullet & \bullet & \bullet & \bullet & \bullet & \bullet & \bullet & \bullet & 0 & 0 & 0 & 0 & 0 & 0 & 0 & 0 & 0 & 0 & 0 & 0 & 0 & 0 & 0 & 0 & 0 & 0 & 0 & 0 \\[.25em]
x_4x_4 & \bullet & 0 & 0 & 0 & 0 & \bullet & \bullet & \bullet & \bullet & \bullet & \bullet & \bullet & \bullet & \bullet & \bullet & 0 & 0 & 0 & 0 & 0 & 0 & 0 & 0 & 0 & 0 & 0 & 0 & 0 & 0 & 0 & 0 & 0 & 0 & 0 & 0 \\[.25em]
x_1x_1x_1  & 0 & \bullet & \bullet & \bullet & \bullet & 0 & 0 & 0 & 0 & 0 & 0 & 0 & 0 & 0 & 0 & \bullet & \bullet & \bullet & \bullet & \bullet & \bullet & \bullet & \bullet & \bullet & \bullet & \bullet & \bullet & \bullet & \bullet & \bullet & \bullet & \bullet & \bullet & \bullet & \bullet \\[.25em]
x_1x_1x_2  & 0 & \bullet & \bullet & \bullet & \bullet & 0 & 0 & 0 & 0 & 0 & 0 & 0 & 0 & 0 & 0 & \bullet & \bullet & \bullet & \bullet & \bullet & \bullet & \bullet & \bullet & \bullet & \bullet & \bullet & \bullet & \bullet & \bullet & \bullet & \bullet & \bullet & \bullet & \bullet & \bullet \\[.25em]
x_1x_1x_3  & 0 & \bullet & \bullet & \bullet & \bullet & 0 & 0 & 0 & 0 & 0 & 0 & 0 & 0 & 0 & 0 & \bullet & \bullet & \bullet & \bullet & \bullet & \bullet & \bullet & \bullet & \bullet & \bullet & \bullet & \bullet & \bullet & \bullet & \bullet & \bullet & \bullet & \bullet & \bullet & \bullet \\[.25em]
x_1x_1x_4  & 0 & \bullet & \bullet & \bullet & \bullet & 0 & 0 & 0 & 0 & 0 & 0 & 0 & 0 & 0 & 0 & \bullet & \bullet & \bullet & \bullet & \bullet & \bullet & \bullet & \bullet & \bullet & \bullet & \bullet & \bullet & \bullet & \bullet & \bullet & \bullet & \bullet & \bullet & \bullet & \bullet \\[.25em]
x_1x_2x_2  & 0 & \bullet & \bullet & \bullet & \bullet & 0 & 0 & 0 & 0 & 0 & 0 & 0 & 0 & 0 & 0 & \bullet & \bullet & \bullet & \bullet & \bullet & \bullet & \bullet & \bullet & \bullet & \bullet & \bullet & \bullet & \bullet & \bullet & \bullet & \bullet & \bullet & \bullet & \bullet & \bullet \\[.25em]
x_1x_2x_3  & 0 & \bullet & \bullet & \bullet & \bullet & 0 & 0 & 0 & 0 & 0 & 0 & 0 & 0 & 0 & 0 & \bullet & \bullet & \bullet & \bullet & \bullet & \bullet & \bullet & \bullet & \bullet & \bullet & \bullet & \bullet & \bullet & \bullet & \bullet & \bullet & \bullet & \bullet & \bullet & \bullet \\[.25em]
x_1x_2x_4  & 0 & \bullet & \bullet & \bullet & \bullet & 0 & 0 & 0 & 0 & 0 & 0 & 0 & 0 & 0 & 0 & \bullet & \bullet & \bullet & \bullet & \bullet & \bullet & \bullet & \bullet & \bullet & \bullet & \bullet & \bullet & \bullet & \bullet & \bullet & \bullet & \bullet & \bullet & \bullet & \bullet \\[.25em]
x_1x_3x_3  & 0 & \bullet & \bullet & \bullet & \bullet & 0 & 0 & 0 & 0 & 0 & 0 & 0 & 0 & 0 & 0 & \bullet & \bullet & \bullet & \bullet & \bullet & \bullet & \bullet & \bullet & \bullet & \bullet & \bullet & \bullet & \bullet & \bullet & \bullet & \bullet & \bullet & \bullet & \bullet & \bullet \\[.25em]
x_1x_3x_4  & 0 & \bullet & \bullet & \bullet & \bullet & 0 & 0 & 0 & 0 & 0 & 0 & 0 & 0 & 0 & 0 & \bullet & \bullet & \bullet & \bullet & \bullet & \bullet & \bullet & \bullet & \bullet & \bullet & \bullet & \bullet & \bullet & \bullet & \bullet & \bullet & \bullet & \bullet & \bullet & \bullet \\[.25em]
x_1x_4x_4  & 0 & \bullet & \bullet & \bullet & \bullet & 0 & 0 & 0 & 0 & 0 & 0 & 0 & 0 & 0 & 0 & \bullet & \bullet & \bullet & \bullet & \bullet & \bullet & \bullet & \bullet & \bullet & \bullet & \bullet & \bullet & \bullet & \bullet & \bullet & \bullet & \bullet & \bullet & \bullet & \bullet \\[.25em]
x_2x_2x_2  & 0 & \bullet & \bullet & \bullet & \bullet & 0 & 0 & 0 & 0 & 0 & 0 & 0 & 0 & 0 & 0 & \bullet & \bullet & \bullet & \bullet & \bullet & \bullet & \bullet & \bullet & \bullet & \bullet & \bullet & \bullet & \bullet & \bullet & \bullet & \bullet & \bullet & \bullet & \bullet & \bullet \\[.25em]
x_2x_2x_3  & 0 & \bullet & \bullet & \bullet & \bullet & 0 & 0 & 0 & 0 & 0 & 0 & 0 & 0 & 0 & 0 & \bullet & \bullet & \bullet & \bullet & \bullet & \bullet & \bullet & \bullet & \bullet & \bullet & \bullet & \bullet & \bullet & \bullet & \bullet & \bullet & \bullet & \bullet & \bullet & \bullet \\[.25em]
x_2x_2x_4  & 0 & \bullet & \bullet & \bullet & \bullet & 0 & 0 & 0 & 0 & 0 & 0 & 0 & 0 & 0 & 0 & \bullet & \bullet & \bullet & \bullet & \bullet & \bullet & \bullet & \bullet & \bullet & \bullet & \bullet & \bullet & \bullet & \bullet & \bullet & \bullet & \bullet & \bullet & \bullet & \bullet \\[.25em]
x_2x_3x_3  & 0 & \bullet & \bullet & \bullet & \bullet & 0 & 0 & 0 & 0 & 0 & 0 & 0 & 0 & 0 & 0 & \bullet & \bullet & \bullet & \bullet & \bullet & \bullet & \bullet & \bullet & \bullet & \bullet & \bullet & \bullet & \bullet & \bullet & \bullet & \bullet & \bullet & \bullet & \bullet & \bullet \\[.25em]
x_2x_3x_4  & 0 & \bullet & \bullet & \bullet & \bullet & 0 & 0 & 0 & 0 & 0 & 0 & 0 & 0 & 0 & 0 & \bullet & \bullet & \bullet & \bullet & \bullet & \bullet & \bullet & \bullet & \bullet & \bullet & \bullet & \bullet & \bullet & \bullet & \bullet & \bullet & \bullet & \bullet & \bullet & \bullet \\[.25em]
x_2x_4x_4  & 0 & \bullet & \bullet & \bullet & \bullet & 0 & 0 & 0 & 0 & 0 & 0 & 0 & 0 & 0 & 0 & \bullet & \bullet & \bullet & \bullet & \bullet & \bullet & \bullet & \bullet & \bullet & \bullet & \bullet & \bullet & \bullet & \bullet & \bullet & \bullet & \bullet & \bullet & \bullet & \bullet \\[.25em]
x_3x_3x_3  & 0 & \bullet & \bullet & \bullet & \bullet & 0 & 0 & 0 & 0 & 0 & 0 & 0 & 0 & 0 & 0 & \bullet & \bullet & \bullet & \bullet & \bullet & \bullet & \bullet & \bullet & \bullet & \bullet & \bullet & \bullet & \bullet & \bullet & \bullet & \bullet & \bullet & \bullet & \bullet & \bullet \\[.25em]
x_3x_3x_4  & 0 & \bullet & \bullet & \bullet & \bullet & 0 & 0 & 0 & 0 & 0 & 0 & 0 & 0 & 0 & 0 & \bullet & \bullet & \bullet & \bullet & \bullet & \bullet & \bullet & \bullet & \bullet & \bullet & \bullet & \bullet & \bullet & \bullet & \bullet & \bullet & \bullet & \bullet & \bullet & \bullet \\[.25em]
x_3x_4x_4  & 0 & \bullet & \bullet & \bullet & \bullet & 0 & 0 & 0 & 0 & 0 & 0 & 0 & 0 & 0 & 0 & \bullet & \bullet & \bullet & \bullet & \bullet & \bullet & \bullet & \bullet & \bullet & \bullet & \bullet & \bullet & \bullet & \bullet & \bullet & \bullet & \bullet & \bullet & \bullet & \bullet \\[.25em]
x_4x_4x_4  & 0 & \bullet & \bullet & \bullet & \bullet & 0 & 0 & 0 & 0 & 0 & 0 & 0 & 0 & 0 & 0 & \bullet & \bullet & \bullet & \bullet & \bullet & \bullet & \bullet & \bullet & \bullet & \bullet & \bullet & \bullet & \bullet & \bullet & \bullet & \bullet & \bullet & \bullet & \bullet & \bullet 
\end{array}
$$
\textit{Note that there are four diagonal blocks in the complex moment matrix, while there are two diagonal blocks in the real moment matrix (after permutation).}
\end{example} 

The above example illustrates a block structure in the Lasserre hierarchy which can be generalized. At order $d$, there are $d+1$ blocks in the complex moment matrix, as opposed to 2 blocks in the real moment matrix (regardless of the relaxation order). This can easily be deduced from the following considerations.

If $\sigma(z,\bar{z}) = \sum_{\alpha,\beta} \sigma_{\alpha,\beta} z^\alpha \bar{z}^\beta$ is a Hermitian sum-of-squares, then its associated \textit{balanced form}, i.e. $ \sum_{|\alpha|=|\beta|} \sigma_{\alpha,\beta} z^\alpha \bar{z}^\beta$, is also a Hermitian sum-of-squares. This follows readily from the \textit{homogeneous decomposition}
\begin{equation}
\sum_{|\alpha|=|\beta|} \sigma_{\alpha,\beta} z^\alpha \bar{z}^\beta ~~~ = ~~~ \sum_{h = 0}^d ~~ \sum_k  \left| \sum_{|\alpha| = h} p_{k,\alpha} z^{\alpha} \right|^2
\end{equation}
where $\sigma(z,\bar{z}) ~=~ \sum_k \left| \sum_{|\alpha| \leqslant d} p_{k,\alpha} z^\alpha \right|^2$. The homogeneous decomposition becomes relevant when the objective $f$ and constraints $g_1,\hdots, g_m$ are themselves balanced forms. Indeed, we then get the following property. If $(\sigma_0,\hdots,\sigma_m)$ is a feasible point of the relaxation of order $d$, that is to say $f - \lambda = \sigma_0 + \sigma_1 g_1 + \hdots + \sigma_m g_m$, then the associated balanced forms also constitute a feasible point. In other words, the non-balanced terms $\sigma_{\alpha,\beta} z^\alpha \bar{z}^\beta, ~ |\alpha| \neq |\beta|,$ can be discarded. The homogeneous decomposition accounts for the $(d+1)$-block diagonal structure of the complex Lasserre hierarchy.

Likewise, if $\sigma(x) = \sum_{\alpha} \sigma_{\alpha} x^{\alpha}$ is a real sum-of-squares, then its associated \textit{even form}, i.e. $\sum_{|\alpha|~\text{even}} \sigma_{\alpha} x^{\alpha}$, is also a real sum-of-squares. This follows readily from the \textit{even/odd decomposition}
\begin{equation}
\sum_{|\alpha|~\text{even}} \sigma_{\alpha} x^{\alpha} ~~~ = ~~~  \sum_k \left( \sum_{|\alpha|~\text{even}} p_{k,\alpha} x^{\alpha} \right)^2 ~+~ \left( \sum_{|\alpha|~\text{odd}} p_{k,\alpha} x^{\alpha} \right)^2
\end{equation}
where $
\sigma(x) ~=~ \sum_k \left( \sum_{\alpha} p_{k,\alpha} x^\alpha \right)^2$. The even/odd decomposition becomes relevant when the objective $f$ and constraints $g_1,\hdots, g_m$ are themselves even forms. If $(\sigma_0,\hdots,\sigma_m)$ is a feasible point of the relaxation of order $d$, then the associated even forms also constitute a feasible point. The even/odd decomposition accounts for the 2-block diagonal structure of the real Lasserre hierarchy.

We next analyse the above results via a dual perspective based on measure theory. We illustrate it with Figure \ref{fig:grasp} which we've generated with MATLAB and Paint.
\begin{figure}[!h]
\centering
\begin{minipage}{.62\textwidth}
  \centering
  \includegraphics[width=1.1\textwidth]{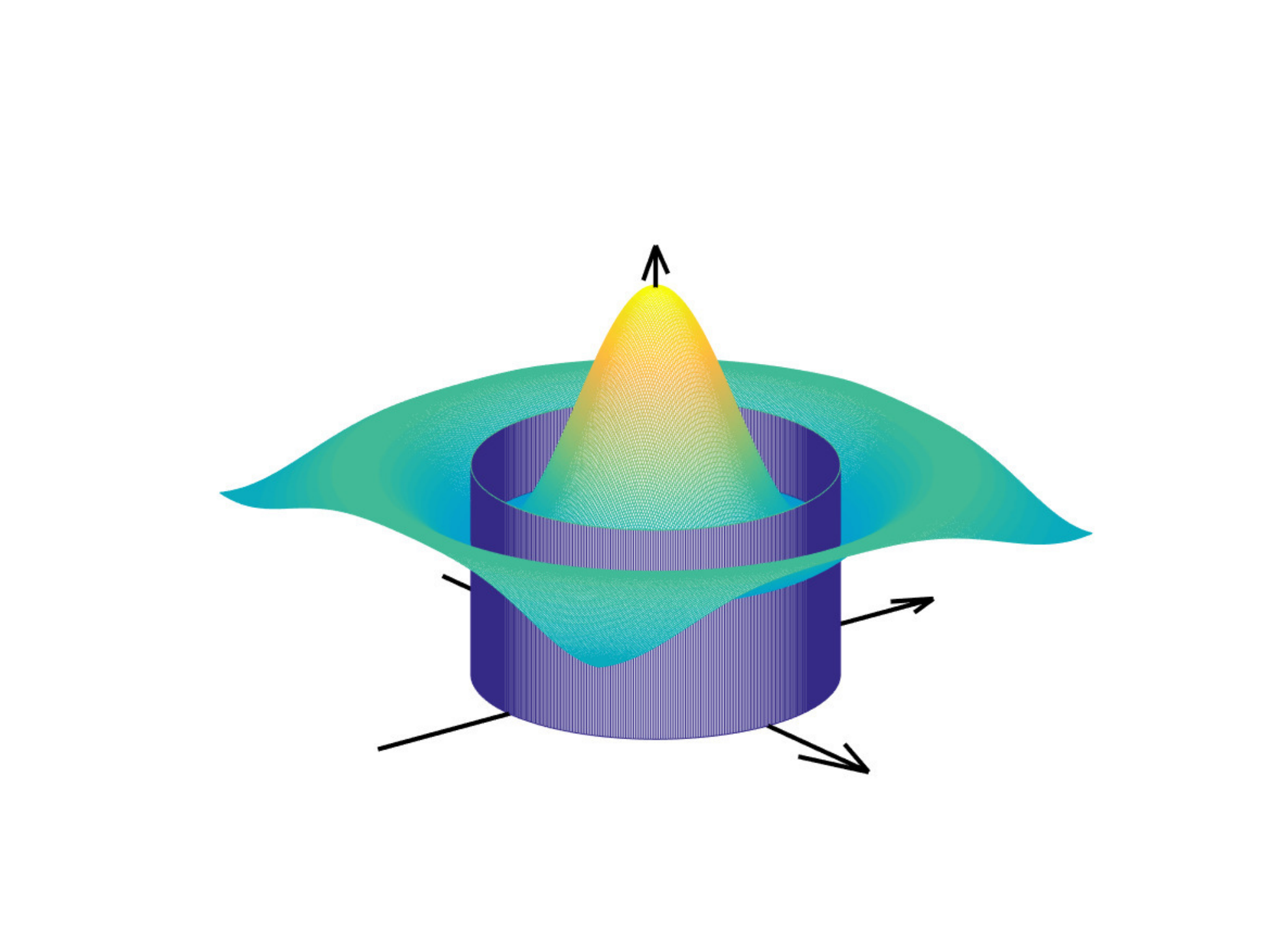}
  \label{fig:test1}
\end{minipage}%
\begin{minipage}{.38\textwidth}
  \centering
  \includegraphics[width=1.1\textwidth]{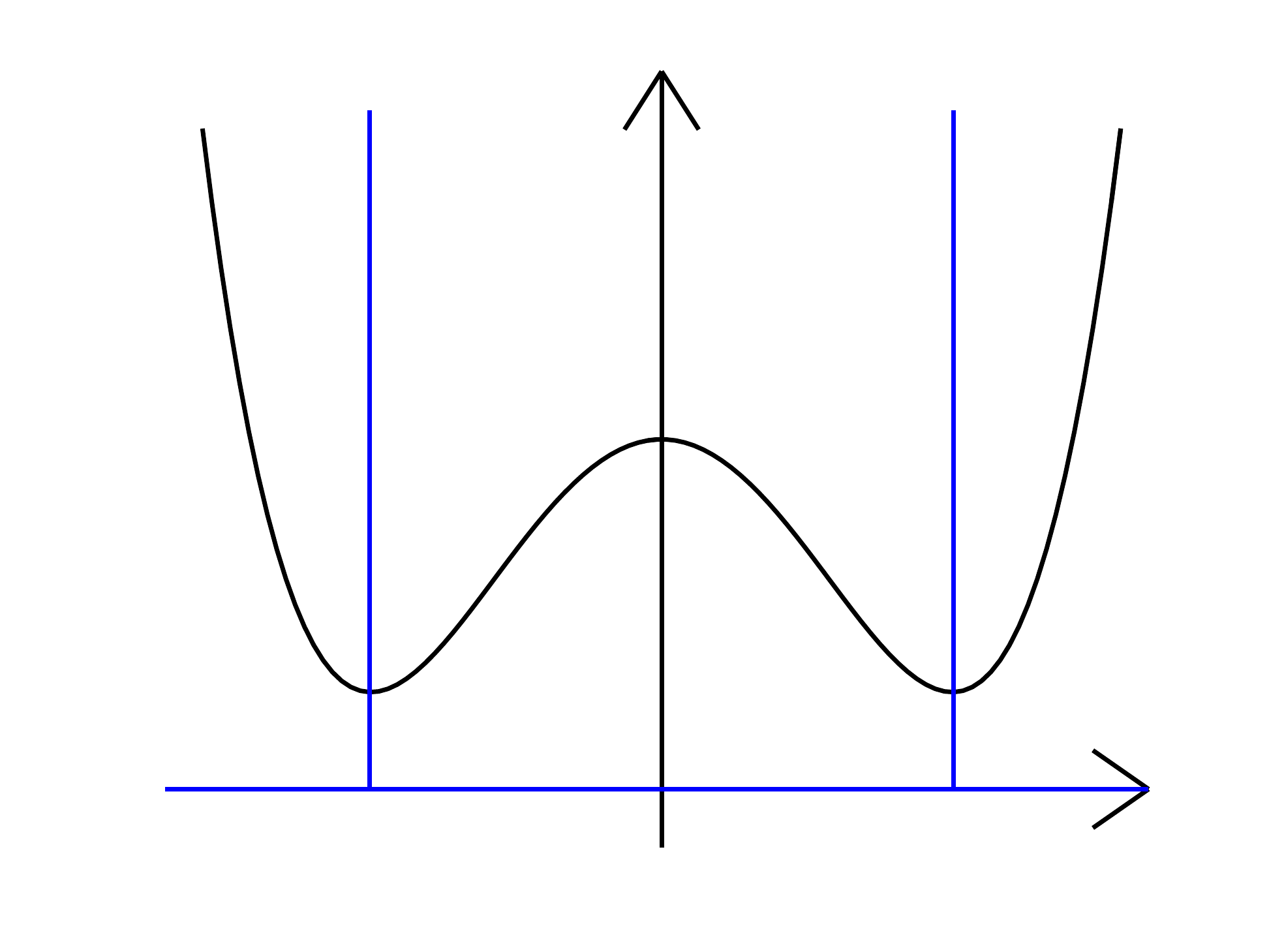}
  \label{fig:test2}
\end{minipage}
\caption{Invariant measure on the complex plane versus on the real line}
\label{fig:grasp}
\end{figure}

On the left, we seek to minimize a function $f(z)$ of one unconstrained complex variable. The two bottom axis correspond to the real and imaginary parts of the variable respectively. The objective function is invariant under the action of the torus $\mathbb{T}$ (i.e. $f(e^{i\theta}z) = f(z)$). Thus, one may seek a measure that is also invariant, instead of looking for a Dirac measure. The cylinder represents an invariant measure $\mu$ minimizing $\int f d\mu$. Such an invariant measure satisfies $\int z^\alpha \bar{z}^\beta d\mu = 0 ~,~ \text{if} ~ |\alpha| \neq |\beta|$. In order to seek such a measure, one may therefore set the corresponding pseudo-moments $y_{\alpha,\beta}$ to zero in the complex moment matrix. 

Seeking an invariant measure under the action of a finite group in the real Lasserre hierarchy was proposed in \cite{riener-2013}. Above, we applied this idea to the complex Lasserre hierarchy for a compact group (the torus). We next apply it to the real Lasserre hierarchy regarding a symmetry not considered in \cite{riener-2013}, namely $\{-1,+1\}$. In particular, that work does not exhibit a block diagonal structure of the Lasserre hierarchy. Note also that a general theory of invariance in sums of squares was developed in~\cite{cimpric-2009}. However, the two symmetries that we consider are not studied.

On the right of Figure \ref{fig:grasp}, we seek to minimize a function $f(x)$ of one unconstrained real variable. It is invariant under the action of the symmetry group $\{-1,+1\}$ (i.e. $f(-x) = f(x)$). Thus one may seek a measure that is also invariant. The vertical lines represent an invariant measure $\mu$ minimizing $\int f d\mu$. Such an invariant measure satisfies $\int x^{\alpha} d\mu = 0 ~,~ \text{if} ~ |\alpha| ~ \text{is odd}$. In order to seek such a measure, one may therefore set the corresponding pseudo-moments $y_{\alpha}$ to zero in the real moment matrix.

\section{Conclusion}
\label{sec:Conclusion}
To summarize, we propose three notions to handle large scale polynomial optimization problems: 1) a \textit{complex Lasserre hierarchy} which generalizes the theory of Lasserre to complex numbers; 2) a \textit{multi-ordered Lasserre hierarchy} to exploit sparsity in real or complex variables by associating a relaxation order to each constraint; 3) a \textit{block diagonal Lasserre hierarchy} to exploit symmetry in real or complex variables. We apply the three notions to the optimal power flow problem in electrical engineering. To the best of our knowledge, the Lasserre hierarchy was previously limited to small scale problems, while we solve a large scale industrial problem with thousands of variables and constraints to global optimality.

\section*{Acknowledgements}
We wish to thank the anonymous reviewers for their precious time and valuable feedback.
Special thanks to Mihai Putinar for the fruitful discussions that helped us to improve this paper. We also wish to thank Jean-Bernard Baillon, Didier Henrion, Jean Bernard Lasserre, Bruno Nazaret, and Markus Schweighofer for their insightful comments. The first author would also like to thank Nicolas Breuillac and Robin Jelin for fruitful discussions during his visit to Berlin.


\appendix

\section{Proof of Theorem \ref{th:hypo}}
\label{app:hypo}

($\Longleftarrow$) The positive semidefinite moment matrix of rank $r:=\text{rank} M_{d+d_K}(y)$ can be factorized in Grammian form as $y_{\alpha,\beta} = x_{\alpha}^* x_{\beta}$, for all $|\alpha|,|\beta|\leqslant d+d_K,$ where $x_\alpha \in \mathbb{C}^r$. This leads us to consider the finite dimensional Hilbert space $\mathbb{C}^r = \text{span} ( x_{\alpha} )_{|\alpha|\leqslant d+d_K} = \text{span} ( x_{\alpha} )_{|\alpha|\leqslant d}$, the last equality being a consequence of $\text{rank} M_{d+d_K}(y) = \text{rank} M_{d}(y)$.\footnote{Indeed, consider $d<|\beta|\leqslant d+d_K$. The column of $M_{d+d_K}(y)$ indexed by $\beta$ is a linear combination of the columns of $M_{d+d_K}(y)$ indexed by $\alpha$ with $|\alpha|\leqslant d$. In other words, there exists some complex numbers $(c_{\alpha})_{|\alpha|\leqslant d}$ such that 
$ y_{\gamma,\beta} = \sum_{|\alpha|\leqslant d} c_{\alpha}  y_{\gamma,\alpha},~ \forall |\gamma| \leqslant d+d_K$. As a result, $x_{\gamma}^* x_{\beta} = \sum_{|\alpha|\leqslant d} c_{\alpha} x_{\gamma}^* x_{\alpha},~ \forall |\gamma| \leqslant d+d_K$. To conclude, $ x_{\beta} -  \sum_{|\alpha|\leqslant d} c_{\alpha} x_{\alpha}  \in  \left( \text{span} ( x_{\alpha} )_{|\alpha|\leqslant d+d_K}\right)^\perp \cap \left(  \text{span} ( x_{\alpha} )_{|\alpha|\leqslant d+d_K} \right)  =  \{ 0 \}$, where $(\cdot)^\perp$ stands for orthogonal.} Since $d_K \geqslant 1$, it is true in particular that $\mathbb{C}^r = \text{span} ( x_{\alpha} )_{|\alpha|\leqslant d+d_K-1}$. On this space, we define the \textit{shift operators} $T_1,\hdots,T_n$ as
\begin{equation}
\begin{array}{rccc}
T_k : & \mathbb{C}^r & \longrightarrow & \mathbb{C}^r \\[.1cm]
         & \sum\limits_{|\alpha| \leqslant d+d_K-1} u_{\alpha} x_{\alpha} & \longmapsto & \sum\limits_{|\alpha| \leqslant d+d_K-1} u_{\alpha} x_{\alpha+e_k}
\end{array}
\end{equation}
where $e_k$ is the row vector of size $n$ that contains only zeros apart from 1 in position $k$. In order to make sure that the shifts are well-defined, we must check that each element of $\mathbb{C}^r$ has a unique image by $T_k$. In other words, given two sets of coefficients $(u_{\alpha})_{|\alpha|\leqslant d+d_K-1}$ and $(v_{\alpha})_{|\alpha|\leqslant d+d_K-1}$, if $\sum_{|\alpha| \leqslant d+d_K-1} u_\alpha x_{\alpha} = \sum_{|\alpha| \leqslant d+d_K-1} v_\alpha x_{\alpha}$, then it must be that $\sum_{|\alpha| \leqslant d+d_K-1} u_\alpha x_{\alpha+e_k} = \sum_{|\alpha| \leqslant d+d_K-1} v_\alpha x_{\alpha+e_k}$. Indeed, this is true because
\begin{equation}
\left\| \sum\limits_{|\alpha|\leqslant d+d_K-1} (u_{\alpha}-v_{\alpha}) x_{\alpha+e_k} \right\| ~~ \leqslant ~~  R ~ \left\| \sum\limits_{|\alpha|\leqslant d+d_K-1} (u_{\alpha}-v_{\alpha}) x_{\alpha} \right\|
\end{equation}
where $\| x \| := \sqrt{x^*x}$ denotes the 2-norm of a vector $x \in \mathbb{C}^r$.
We now explain why the above inequality holds. Given some complex numbers $(w_{\alpha})_{|\alpha|\leqslant d+d_K-1}$, the positivity of the localizing matrix associated to the ball constraint, i.e. $M_{d+d_K-1}[(R^2\!-\!|z_1|^2 \!-\! \hdots\! -\! |z_n|^2)y] \succcurlyeq 0$, implies that
\begin{equation}
\begin{array}{rcl}
\left\| \sum\limits_{|\alpha|\leqslant d+d_K-1} w_{\alpha} x_{\alpha+e_k} \right\|^2 & ~=~ & \sum\limits_{|\alpha|,|\beta|\leqslant d+d_K-1} ~~ \overline{w}_{\alpha} w_{\beta} ~ x_{\alpha+e_k}^*x_{\beta+e_k}  \\[.5cm]
 & ~=~ & \sum\limits_{|\alpha|,|\beta|\leqslant d+d_K-1} ~~ \overline{w}_{\alpha} w_{\beta} ~ y_{\alpha+e_k,\beta+e_k}  \\[.5cm]
 & ~\leqslant~ & R^2 \sum\limits_{|\alpha|,|\beta|\leqslant d+d_K-1} ~~ \overline{w}_{\alpha} w_{\beta} ~ y_{\alpha+e_k,\beta+e_k}  \\[.5cm]
& \leqslant & R^2 \left\| \sum\limits_{|\alpha|\leqslant d+d_K-1} w_{\alpha} x_{\alpha+e_k} \right\|^2.
\end{array}
\end{equation}
We now proceed to show that $T_1,\hdots,T_n,T_1^*,\hdots,T_n^*$ commute pair-wise. When $\text{rank} M_{d+d_K}(y) = \text{rank} M_{d}(y) = 1$, this is trivial since $T_1,\hdots,T_n$ are then a set of complex numbers. Otherwise, we use that $d_K \geqslant 2$ to prove that $T_1,\hdots,T_n$ commute pair-wise. Indeed, for all $|\alpha|\leqslant d \leqslant d+d_K-2$, it holds that  $T_i T_j x_{\alpha} = T_i x_{\alpha+e_j} = x_{\alpha+e_j+e_i} = x_{\alpha+e_i+e_j} = T_j T_i x_{\alpha}$. As a result, given $u \in \mathbb{C}^r$, say with decomposition $u = \sum_{|\alpha| \leqslant d} u_{\alpha} x_{\alpha}$, we have that
\begin{equation} T_i T_j u = T_i T_j \left( \sum_{|\alpha| \leqslant d} u_{\alpha} x_{\alpha} \right) = \sum_{|\alpha| \leqslant d} u_{\alpha} T_i T_j x_{\alpha} = \sum_{|\alpha| \leqslant d} u_{\alpha} T_j T_i x_{\alpha} = T_j T_i u. \end{equation}
We go on to prove the stronger property that $T_1,\hdots,T_n,T_1^*,\hdots,T_n^*$ commute pair-wise. Consider $u,v,w \in \mathbb{C}^r$ admitting the following decompositions
\begin{equation}
\begin{array}{lclcl}
u & = & \sum\limits_{|\alpha|\leqslant d} u_\alpha x_\alpha &~,~& \vec{u} := (u_\alpha)_{|\alpha|\leqslant d} \\[.5cm]
v & = & \sum\limits_{|\alpha|\leqslant d} v_\alpha x_\alpha &~,~& \vec{v} := (v_\alpha)_{|\alpha|\leqslant d} \\[.5cm]
w & = & \sum\limits_{|\alpha|\leqslant d} w_\alpha x_\alpha &~,~& \vec{w} := (w_\alpha)_{|\alpha|\leqslant d} 
\end{array}
\end{equation}
A simple computation (details below) yields that, for all $1\leqslant i < j \leqslant n$,
\begin{equation}
\begin{array}{c}
\begin{pmatrix}
u \\
v \\
w
\end{pmatrix}^*
\begin{pmatrix}
I        & T_i^* & T_j^* \\
T_i    & T_i^* T_i  & T_j^* T_i  \\
T_j    & T_i^* T_j  & T_j^* T_j
\end{pmatrix} 
\begin{pmatrix}
u \\
v \\
w
\end{pmatrix}
~~~ = ~~~ \hdots \\[1cm]
\begin{pmatrix}
\vec{u} \\
\vec{v} \\
\vec{w}
\end{pmatrix}^*
\begin{pmatrix} 
M_{d}(y) & M_{d}(z_i y) & M_{d}(z_j y) \\ 
M_{d}(\bar{z}_i y) & M_{d}(|z_i|^2 y)  & M_{d}(z_j \bar{z}_i y) \\
M_{d}(\bar{z}_j y) & M_{d}(z_i \bar{z}_j y) & M_{d}(|z_j|^2 y)
\end{pmatrix}
\begin{pmatrix}
\vec{u} \\
\vec{v} \\
\vec{w}
\end{pmatrix}
\end{array}
\end{equation}
The ``joint hyponormality of the shifts'' condition then implies that
\begin{equation}
\begin{pmatrix}
I & T_i^* & T_j^* \\
T_i & T_i^*T_i & T_i^*T_j  \\
T_j & T_i^*T_j & T_j^*T_j
\end{pmatrix} \succcurlyeq 0.
\end{equation}
We now dwell on the computational details. We will use the notation $T^{\alpha} := T_1^{\alpha_1} \hdots T_n^{\alpha_n}$ for convenience. For any complex polynomial $g \in \mathbb{C}[z,\bar{z}]$, it holds that $u^* g(T^*,T) v = \vec{u}^* M_{d-d_K}( g y ) \vec{v}$ since
\begin{equation}
\begin{array}{rcl}
u^* g(T^*,T) v & = & u^* \left( \sum\limits_{\gamma,\delta} ~ g_{\gamma,\delta} (T^*)^\gamma  T^\delta \right) v \\[.5cm]
& = & \sum\limits_{\gamma,\delta} ~ g_{\gamma,\delta} ~ (T^\gamma u)^*  T^\delta v \\[.5cm]
& = & \sum\limits_{\alpha,\beta,\gamma,\delta} \overline{u}_\alpha v_\beta ~ g_{\gamma,\delta} ~ (T^\gamma x_\alpha)^*  T^\delta x_\beta \\[.5cm]
& = & \sum\limits_{\alpha,\beta,\gamma,\delta} \overline{u}_\alpha v_\beta ~ g_{\gamma,\delta}  ~ x_{\alpha+\gamma}^*  x_{\beta+\delta} \\[.5cm]
& = & \sum\limits_{\alpha,\beta} \overline{u}_\alpha v_\beta \left( \sum\limits_{\gamma,\delta} g_{\gamma,\delta} ~  y_{\alpha+\gamma,\beta+\delta} \right) \\[.6cm]
& = & \vec{u}^* M_{d}( g y ) \vec{v}.
\end{array}
\end{equation}
Let's pursue the proof: in accordance with Section \ref{sec:Truncated moment problem}, it holds that $T_1,\hdots, T_n$ are jointly hyponormal and that $[T_i,T_j] = T_i^* T_j - T_i T_j^* = 0$. Together with the fact that $T_1,\hdots, T_n$ commute pair-wise, we deduce that $T_1,\hdots,T_n,T_1^*,\hdots,T_n^*$ commute pair-wise. The operators must then be simultaneously diagonalizable. In other words, there exists a unitary matrix $P$ such that $
T_k = PD_kP^*,~ k = 1,\hdots,n,$ where $D_k = \text{diag}(d_{k1},\hdots,d_{kr})$ is a diagonal matrix. For all $|\alpha|,|\beta|\leqslant d+d_K$, we thus have
\begin{equation}
\begin{array}{rcl}
y_{\alpha,\beta} & = & x_\alpha^* x_\beta \\[.5em]
& = & (T^\alpha x_0)^* (T^\beta x_0 )\\[.5em]
& = & x_0^* (T^\alpha)^* T^\beta x_0 \\[.5em]
& = & x_0^* (PD^\alpha P^*)^* PD^\beta P^* x_0 \\[.5em]
& = & x_0^* P \overline{D}^\alpha P^*P D^\beta P^* x_0 \\[.5em]
& = & x_0^* P \overline{D}^\alpha D^\beta P^* x_0 \\[.5em]
& = & x_0^* \left( \sum\limits_{j=1}^r p_j \overline{d}_j^\alpha d_j^\beta p_j^* \right) x_0 \\[1em]
& = & \sum\limits_{j=1}^r x_0^* p_j p_j^* x_0 ~ \overline{d}_j^\alpha d_j^\beta \\[1em]
& = & \sum\limits_{j=1}^r |x_0^* p_j|^2 ~ \overline{d}_j^\alpha d_j^\beta
\end{array}
\end{equation}
where $P =: (p_1 \hdots p_r)$ denote the columns of $P$ and $d_j := (d_{1j},\hdots,d_{nj})$.
As a result, eigenvalues of the shift operators yield the support of a measure, and their eigenvectors yield the weights of a measure. Precisely, the measure $\mu = \sum_{j=1}^r  |x_0^* p_j|^2 ~ \delta_{\overline{d}_j}$
satisfies $y_{\alpha,\beta} = \int_{\mathbb{C}^n} z^{\alpha} \bar{z}^\beta d\mu$ for all $|\alpha|,|\beta| \leqslant d+d_K$. In addition, the atoms are distinct and the weights are positive because $r = \text{rank} M_d(y)$. Finally, the measure is supported on $K$ because 
$$
\begin{array}{rcl}
g_i(\overline{d}_j,d_j) & = & p_j^* p_j ~ \sum\limits_{\gamma,\delta} ~ g_{i,\gamma,\delta} ~ \overline{d}_j^\gamma d_j^\delta ~~~~~~~~~ (p_j^* p_j = 1) \\[.4cm]
& = & \sum\limits_{\gamma,\delta} ~ g_{i,\gamma,\delta} ~ (d_j^\gamma p_j)^* (d_j^\delta p_j)
\end{array}$$
$$
\begin{array}{rcl}
\hphantom{g_i(\overline{d}_j)}\hphantom{g_i(\overline{d}_j,d_j)}\hphantom{g_i(\overline{d}_j,d_j)}& = & \sum\limits_{\gamma,\delta} ~ g_{i,\gamma,\delta} ~ (T^\gamma p_j)^* (T^\delta p_j) ~~~~~~ \left( \text{let} ~ p_j =: \sum\limits_{|\alpha| \leqslant d} p_{\alpha} x_{\alpha} \right)\\[.2cm]
& = & \sum\limits_{|\alpha|,|\beta| \leqslant d} \overline{p}_{\alpha} p_{\beta} \left( \sum\limits_{\gamma,\delta} ~ g_{i,\gamma,\delta} ~  (T^{\gamma}x_{\alpha})^* (T^\delta x_{\beta}) \right)  \\[.4cm]
& = & \sum\limits_{|\alpha|,|\beta| \leqslant d} \overline{p}_{\alpha} p_{\beta} \left( \sum\limits_{\gamma,\delta} ~ g_{i,\gamma,\delta} ~ x_{\alpha+\gamma}^* x_{\beta+\delta} \right) \\[.4cm]
& = & \sum\limits_{|\alpha|,|\beta| \leqslant d} \overline{p}_{\alpha} p_{\beta} \left( \sum\limits_{\gamma,\delta} ~ g_{i,\gamma,\delta} ~ y_{\alpha+\gamma,\beta+\delta} \right) ~ \geqslant ~ 0.
\end{array}
$$
The above inequality is a consequence of $M_{d+d_K-k_i}(g_i y ) \succcurlyeq 0$ and $d \leqslant d+d_K - k_i$.

($\Longrightarrow$) Consider the natural extension given by $y_{\alpha,\beta} = \int_{\mathbb{C}^n} z^\alpha \bar{z}^\beta d \mu$ for all $d < |\alpha|,|\beta| \leqslant d+d_K$. The positivity of the moment matrix follows from the positivity of the weights of the atomic measure. The positivity of the localizing matrices follows from the inclusion of the support of the measure in $K$. The rank is preserved because the rank of the moment matrix cannot exceed the number of atoms. Finally, we have
\begin{equation}
\begin{array}{c}
\begin{pmatrix}
\vec{u} \\
\vec{v} \\
\vec{w}
\end{pmatrix}^*
\begin{pmatrix} 
M_{d}(y) & M_{d}(z_i y) & M_{d}(z_j y) \\ 
M_{d}(\bar{z}_i y) & M_{d}(|z_i|^2 y)  & M_{d}(z_j \bar{z}_i y) \\
M_{d}(\bar{z}_j y) & M_{d}(z_i \bar{z}_j y) & M_{d}(|z_j|^2 y)
\end{pmatrix}
\begin{pmatrix}
\vec{u} \\
\vec{v} \\
\vec{w}
\end{pmatrix} ~~ = ~~ \hdots \\[1cm]
\int_{\mathbb{C}^n} 
\left|u(z) + z_i v(z) + z_j w(z)\right|^2 d \mu
~~ \geqslant ~~ 0
\end{array}
\end{equation}
where $u(z):= \sum\limits_{|\alpha|\leqslant d} u_\alpha z^\alpha$, ~ $v(z) := \sum\limits_{|\alpha|\leqslant d} v_\alpha z^\alpha$, ~ and ~ $w(z) := \sum\limits_{|\alpha|\leqslant d} w_\alpha z^\alpha$.

\section{Proof of Theorem \ref{th:special}}
\label{app:special}

($\Longrightarrow$) This part is identical to the proof of Theorem  \ref{th:hypo}.

($\Longleftarrow$) Just like in the proof of Theorem  \ref{th:hypo}, it holds that $T_1,\hdots,T_n$ are pair-wise commuting. There are two points that need to be addressed: 1) the existence of the shift operators and 2) the pair-wise commutativity of the operators $T_1,\hdots,T_n,T_1^*,\hdots,T_n^*$. To address them, we make use of well-known properties on shift operators (namely unitary and self-adjoint, see \cite[p. 319]{atzmon-1975} for instance).

$\bullet$ $K$ \textit{contains the constraints} $|z_k|^2 = 1, ~ k = 1,\hdots,n$:
The localizing matrix associated to $|z_k|^2 = 1$ is equal to zero, that is $M_{d+d_K-1}[(1-|z_k|^2)y] = 0$. As a result, for all complex numbers $(w_{\alpha})_{|\alpha| \leqslant d+d_K-1}$, it holds that
\begin{equation}
\left\| \sum\limits_{|\alpha| \leqslant d+d_K-1} w_\alpha x_{\alpha+e_k} \right\| ~~=~~ \left\| \sum\limits_{|\alpha| \leqslant d+d_K-1} w_\alpha x_{\alpha} \right\|
\end{equation}
The shifts are thus well-defined.
In addition, for all $|\alpha|,|\beta| \leqslant d$, we have that
\begin{equation} x_{\alpha}^* T_k^*T_k x_{\beta} = (T_k x_{\alpha} )^* (T_k x_{\beta} ) = x_{\alpha+e_k}^* x_{\beta+e_k} = y_{\alpha+e_k,\beta+e_k} = y_{\alpha,\beta} = x_{\alpha}^* x_{\beta}. \end{equation}
As a result, given $u \in \mathbb{C}^r$, say with decomposition $u = \sum_{|\alpha| \leqslant d} u_{\alpha} x_{\alpha}$, we have that
\begin{equation} u^* T_k^*T_k u ~=~ \sum_{|\alpha|,|\beta| \leqslant d} \overline{u}_{\alpha} u_{\beta} x_{\alpha}^* T_k^*T_k x_{\beta} ~=~  \sum_{|\alpha|,|\beta| \leqslant d} \overline{u}_{\alpha} u_{\beta} x_{\alpha}^* x_{\beta} ~ = ~u^* u. \end{equation}
Hence $T_k^*T_k$ is the identity matrix; in other words, the shift operators are unitary. This means that $(T_1,\hdots,T_n,T_1^*,\hdots,T_n^*) = (T_1,\hdots,T_n,T_1^{-1},\hdots,T_n^{-1})$ is a pair-wise commuting tuple of operators. Indeed, if two invertible square matrices $A$ and $B$ commute, so do $A^{-1}$ and $B^{-1}$ (since $A^{-1}B^{-1} AB B^{-1}A^{-1} =  A^{-1}B^{-1} BA B^{-1}A^{-1}$), and so do $A$ and $B^{-1}$ (since $B^{-1} AB B^{-1} =  B^{-1} BA B^{-1}$).

$\bullet$ $K$ \textit{contains the constraints} $\textbf{i}z_k - \textbf{i}\bar{z}_k = 0, ~ k = 1,\hdots,n$:
Consider two sets of complex numbers $(u_{\alpha})_{|\alpha|\leqslant d+d_K-1}$ and $(v_{\alpha})_{|\alpha|\leqslant d+d_K-1}$ and assume that $\sum_{|\alpha| \leqslant d+d_K-1} u_\alpha x_{\alpha} = \sum_{|\alpha| \leqslant d+d_K-1} v_\alpha x_{\alpha}$. We next demonstrate that $\sum_{|\alpha| \leqslant d+d_K-1} u_\alpha x_{\alpha+e_k} = \sum_{|\alpha| \leqslant d+d_K-1} v_\alpha x_{\alpha+e_k}$. To do so, define $w_{\alpha} := u_{\alpha} - v_{\alpha}$ for all $|\alpha|\leqslant d+d_K-1$. For all $|\beta| \leqslant d+d_K-1$, it holds that 
\begin{equation}
\begin{array}{rcl}
x_{\beta}^* \left( \sum\limits_{|\alpha| \leqslant d+d_K-1} w_{\alpha} x_{\alpha+e_k} \right) & = & \sum\limits_{|\alpha| \leqslant d+d_K-1} w_{\alpha} ~ x_{\beta}^* x_{\alpha+e_k} \\[.6cm]
& = & \sum\limits_{|\alpha| \leqslant d+d_K-1} w_{\alpha} ~ y_{\beta,\alpha+e_k} \\[.6cm]
& = & \sum\limits_{|\alpha| \leqslant d+d_K-1} w_{\alpha} ~ y_{\beta+e_k,\alpha} \\[.4cm]
& = & x_{\beta+e_k}^* \left( \sum\limits_{|\alpha| \leqslant d+d_K-1} w_{\alpha} x_{\alpha} \right) ~ = ~ 0
\end{array}
\end{equation}
Since $\mathbb{C}^r = \text{span} (x_{\alpha})_{|\alpha| \leqslant d+d_K-1}$, we conclude that $ \sum_{|\alpha| \leqslant d+d_K-1} w_{\alpha} x_{\alpha+e_k} = 0 $.

Moving on to the latter part of the proof, for all $|\alpha|,|\beta| \leqslant d$, we have that
\begin{equation} x_{\alpha}^* T_k^* x_{\beta} = (T_k x_{\alpha})^*x_{\beta} = x_{\alpha+e_k}^*x_{\beta}  = y_{\alpha+e_k,\beta} = y_{\alpha,\beta+e_k} =  x_{\alpha}^*x_{\beta+e_k} =  x_{\alpha}^* T_k x_{\beta}. \end{equation}
Hence $T_k^* = T_k$; in other words, the shift operators are self-adjoint.
This means that $(T_1,\hdots,T_n,T_1^*,\hdots,T_n^*) = (T_1,\hdots,T_n,T_1,\hdots,T_n)$ is pair-wise commuting.

\bibliography{mybib}{}

\begin{thebibliography}{10}

\bibitem{ieee_test_cases}
{\em {Power Systems Test Case Archive}}, {University of Washington Department
  of Electrical Engineering}
  \url{http://www.ee.washington.edu/research/pstca/}.

\bibitem{li2006}
{\sc X.~Li A. and S.~Ranga}, {\em {Szeg\"{o} Polynomials and the Truncated
  Trigonometric Moment Problem}}, The Ramanujan Journal, 12 (2006),
  pp.~461--472.

\bibitem{ahmadi-2014}
{\sc A.~Ali Ahmadi and A.~Majumdar}, {\em {DSOS and SDSOS Optimization: LP and
  SOCP-based Alternatives to Sum of Squares Optimization}}, 48th Annual
  Conference on Information Sciences and Systems (CISS),  (2014).

\bibitem{aittomaki-2009}
{\sc T.~Aittomaki and V.~Koivunen}, {\em {Beampattern Optimization by
  Minimization of Quartic Polynomial}}, IEEE/SP 15th W. Stat. Signal Process.,
  51 (2009), pp.~437--–440.

\bibitem{akhiezer-1965}
{\sc N.I. Akhiezer}, {\em The Classical Moment Problem and Some Related
  Questions in Analysis}, Hafner Publ. Co., New York, 1965.

\bibitem{akhiezer1962}
{\sc N.~I. Akhiezer and M.~Krein}, {\em {Some Questions in the Theory of
  Moments}}, Transl. Math. Monographs 2, 58 (1962), pp.~164--168.

\bibitem{andersen2014}
{\sc M.S. Andersen, A.~Hansson, and L.~Vandenberghe}, {\em {Reduced-Complexity
  Semidefinite Relaxations of Optimal Power Flow Problems}}, IEEE TPS, 29
  (2014), pp.~1855--1863.

\bibitem{athavale1988}
{\sc A.~Athavale}, {\em {On Joint Hyponormality of Operators}}, Proceedings of
  the American Mathematical Society, 103 (1988).

\bibitem{athavale1990}
{\sc A.~Athavale and S.~Pederson}, {\em {Moment Problems and Subnormality}},
  Journal of Mathematical Analysis and Applications, 146 (1990), pp.~434--441.

\bibitem{atzmon-1975}
{\sc A.~Atzmon}, {\em {A Moment Problem for Positive Measures on the Unit
  Disc}}, Pacific J. Math., 59 (1975), pp.~317--325.

\bibitem{bai-fujisawa-wang-wei-2008}
{\sc X.~Bai, H.~Wei, K.~Fujisawa, and Y.~Wang}, {\em {Semidefinite Programming
  for Optimal Power Flow Problems}}, Int. J. Elec. Power, 30 (2008),
  pp.~383--392.

\bibitem{bakonyi2011}
{\sc M.~Bakonyi and H.~J. Woerdeman}, {\em Matrix Completions, Moments and Sums
  of Hermitian Squares}, Princeton University Press, 2011.

\bibitem{bandeira-2014}
{\sc A.S. Bandeira, N.~Boumal, and A.~Singer}, {\em {Tightness of the Maximum
  Likelihood Semidefinite Relaxation for Angular Synchronization}}, Math.
  Program.,  (2016), pp.~1--–23.

\bibitem{bukhsh2013}
{\sc W.A Bukhsh, A.~Grothey, K.I. {McKinnon}, and P.A. Trodden}, {\em {Local
  Solutions of the Optimal Power Flow Problem}}, IEEE TPS, 28 (2013),
  pp.~4780–--4788.

\bibitem{bukhsh-archive}
\leavevmode\vrule height 2pt depth -1.6pt width 23pt, {\em {Test Case Archive
  of Optimal Power Flow (OPF) Problems with Local Optima}},
  \url{http://www.maths.ed.ac.uk/optenergy/LocalOpt/WB5.html},  (2013).

\bibitem{candes-2013}
{\sc E.J. Cand\`es, Y.~C. Eldar, T.~Strohmer, and V.~Voroninski}, {\em {Phase
  Retrieval via Matrix Completion}}, SIAM J. Imaging Sci., 6 (2013),
  pp.~199–--225.

\bibitem{carpentier-1962}
{\sc M.J. Carpentier}, {\em {Contribution \`a l'\'Etude du Dispatching
  \'Economique}}, Bull. de la Soc. Fran. des \'Elec., 8 (1962),
  pp.~431–--447.

\bibitem{cassier1984}
{\sc G.~Cassier}, {\em {Probl\`eme des moments sur un compact de Rn et
  d\'ecomposition de polyn\^{o}mes \`a plusieurs variables}}, J. Funct. Anal.,
  58 (1984), pp.~254--266.

\bibitem{castillo-2013}
{\sc A.~Castillo and R.P. O'Neill}, {\em {Survey of Approaches to Solving the
  ACOPF (OPF Paper 4)}}, tech. report, US FERC, Mar. 2013.

\bibitem{catlin-1996}
{\sc D.W. Catlin and J.P. D’Angelo}, {\em {A Stabilization Theorem for
  Hermitian Forms and Applications to Holomorphic Mappings}}, Math. Res. Lett.,
  3 (1996), pp.~149--–166.

\bibitem{chen-2009}
{\sc C.~Chen and P.P. Vaidyanathan}, {\em {MIMO Radar Waveform Optimization
  With Prior Information of the Extended Target and Clutter}}, IEEE Trans.
  Signal Process., 57 (2009), pp.~3533--3544.

\bibitem{cimpric-2009}
{\sc J.~Cimpric, S.~Kuhlmann, and C.~Scheiderer}, {\em {Sums of Squares and
  Moment Problems in Equivariant Situations}}, Trans. Am. Math. Soc., 361
  (2009), pp.~735--765.

\bibitem{nesta}
{\sc C.~Coffrin, D.~Gordon, and P.~Scott}, {\em {NESTA, the NICTA energy system
  test case archive}}, \url{arXiv:1411.0359},  (2016).

\bibitem{coffrin2015}
{\sc C.~Coffrin, H.L. Hijazi, and P.~{Van Hentenryck}}, {\em {The QC
  Relaxation: Theoretical and Computational Results on Optimal Power Flow}},
  IEEE TPS, 31 (2016), pp.~3008--3018.

\bibitem{curto1988}
{\sc R.~E. Curto}, {\em {Joint Hyponormality: A Bridge Between Hyponormality
  and Subnormality}}, J. Operator Theory: Operator Algebras and Applications,
  (1988).

\bibitem{curto1991}
{\sc R.~E. Curto and L.~A. Fialkow}, {\em {Recursiveness, Positivity, and
  Truncated Moment Problems}}, Houston J. Of Math., 17 (1991).

\bibitem{curto-1996}
\leavevmode\vrule height 2pt depth -1.6pt width 23pt, {\em {Solution of the
  Truncated Complex Moment Problem for Flat Data}}, Memoirs Amer. Math. Soc.,
  568 (1996).

\bibitem{curto-2000}
\leavevmode\vrule height 2pt depth -1.6pt width 23pt, {\em {The Truncated
  Complex K-Moment Problem}}, Trans. Amer. Math. Soc., 352 (2000),
  pp.~2825--2855.

\bibitem{curto-2005}
\leavevmode\vrule height 2pt depth -1.6pt width 23pt, {\em {Truncated K-Moment
  Problems in Several Variables}}, J. Operator Theory, 54 (2005), pp.~189--226.

\bibitem{curto1993}
{\sc R.~E. Curto and M.~Putinar}, {\em {Nearly Subnormal Operators and Moment
  Problems}}, J. Funct. Anal., 2 (1993), pp.~480–--497.

\bibitem{curto-2010}
\leavevmode\vrule height 2pt depth -1.6pt width 23pt, {\em {Polynomially
  Hyponormal Operators}}, Operator Theory: Advances and Applications, 207
  (2010), pp.~195--207.

\bibitem{angelo-2002}
{\sc J.P. D'Angelo}, {\em {Inequalities from Complex Analysis}}, Carus Math.
  Monogr., MAA, 2002.

\bibitem{angelo-2010}
\leavevmode\vrule height 2pt depth -1.6pt width 23pt, {\em {Hermitian Analogues
  of Hilbert's 17th Problem}}, Adv. Math., 226 (2011), pp.~4607--4637.

\bibitem{angelo-2008}
{\sc J.P. D'Angelo and M.~Putinar}, {\em {Polynomial Optimization on
  Odd-Dimensional Spheres}}, in Emerging Applications of Algebraic Geometry,
  Springer New York, 2008.

\bibitem{putinar-2012}
\leavevmode\vrule height 2pt depth -1.6pt width 23pt, {\em {Hermitian
  Complexity of Real Polynomial Ideals}}, Int. J. Math., 23 (2012).

\bibitem{pegase}
{\sc S.~Fliscounakis, P.~Panciatici, F.~Capitanescu, and L.~Wehenkel}, {\em
  {Contingency Ranking with Respect to Overloads in Very Large Power Systems
  Taking into Account Uncertainty, Preventive and Corrective Actions}}, IEEE
  TPS, 28 (2013), pp.~4909--4917.

\bibitem{gabardo1999}
{\sc J.-P. Gabardo}, {\em {Truncated Trigonometric Moment Problems and
  Determinate Measures}}, J. Math. Anal. Appl., 239 (1999), pp.~349--370.

\bibitem{ibm_paper}
{\sc B.~Ghaddar, J.~Marecek, and M.~Mevissen}, {\em {Optimal Power Flow as a
  Polynomial Optimization Problem}}, IEEE Trans. Power Syst.,  (2015).

\bibitem{cvx}
{\sc M.~Grant and S.~Boyd}, {\em {CVX}: Matlab software for disciplined convex
  programming, version 2.0}.
\newblock \url{http://cvxr.com/cvx}, Aug. 2012.

\bibitem{schweighofer-2007}
{\sc D.~Grimm, T.~Netzer, and M.~Schweighofer}, {\em {A Note on the
  Representation of Positive Polynomials with Structured Sparsity}}, Arch.
  Math., 89 (2007), pp.~399--403.

\bibitem{halmos1950}
{\sc P.R. Halmos}, {\em {Normal Dilations and Extensions of Operators}},
  ￼Summa Bras. Math., 2 (1950), pp.~125--134.

\bibitem{harmouch2017}
{\sc J.~Harmouch, H.~Khalil, and B.~Mourrain}, {\em {Structured Low Rank
  Decomposition of Multivariate Hankel Matrices, hal-01440063}},  (2017).

\bibitem{henrion2005}
{\sc D.~Henrion and J.~B. Lasserre}, {\em {Detecting Global Optimality and
  Extracting Solutions in GloptiPoly}}, Part III Numerical Aspects Of
  Polynomial Positivity: Structures, Positive Polynomials in Control, 312
  (2005), pp.~293--310.

\bibitem{henrion2009}
\leavevmode\vrule height 2pt depth -1.6pt width 23pt, {\em {GloptiPoly 3:
  Moments, Optimization and Semidefinite Programming}}, Optimization Methods
  and Software, 24 (2009), pp.~761--779.

\bibitem{hilling-2010}
{\sc J.J. Hilling and A.~Sudbery}, {\em {The Geometric Measure of Multipartite
  Entanglement and the Singular Values of a Hypermatrix}}, J. Math. Phys., 51
  (2010).

\bibitem{iohvidov1982}
{\sc I.~S. Iohvidov}, {\em Hankel and Toeplitz Matrices and Forms: Algebraic
  Theory}, Birkh\"{a}user Verlag, Boston, 1982.

\bibitem{jabr2006}
{\sc R.A. Jabr}, {\em {Radial Distribution Load Flow using Conic Programming}},
  IEEE TPS, 21 (2006), pp.~1458–--1459.

\bibitem{josz-phd}
{\sc C.~Josz}, {\em {Application of Polynomial Optimization to Electricity
  Transmission Networks (PhD thesis)}},
  \url{https://arxiv.org/pdf/1608.03871v1.pdf },  (2016).

\bibitem{josz2017}
\leavevmode\vrule height 2pt depth -1.6pt width 23pt, {\em {Counterexample to
  Global Convergence of DSOS and SDSOS hierarchies}},
  \url{https://arxiv.org/pdf/1707.02964.pdf},  (2017).

\bibitem{josz-2016}
{\sc C.~Josz, S.~Fliscounakis, J.~Maeght, and P.~Panciatici}, {\em {AC Power
  Flow Data in MATPOWER and QCQP format: iTesla, RTE Snapshots, and PEGASE}},
  \url{https://arxiv.org/abs/1603.01533},  (2016).

\bibitem{josz-2015}
{\sc C.~Josz and D.~Henrion}, {\em {Strong Duality in Lasserre's Hierarchy for
  Polynomial Optimization}}, Springer Optim. Lett.,  (2015).

\bibitem{cedric_tps}
{\sc C.~Josz, J.~Maeght, P.~Panciatici, and J.C. Gilbert}, {\em {Application of
  the Moment-SOS Approach to Global Optimization of the OPF Problem}}, IEEE
  TPS, 30 (2015), pp.~463--470.

\bibitem{kimsey2016}
{\sc D.~P. Kimsey}, {\em {The Subnormal Completion Problem in Several
  Variables}}, J. Math. Anal. Appl., 434 (2016), pp.~1504--–1532.

\bibitem{kimsey2013}
{\sc D.~P. Kimsey and H.~J. Woerdeman}, {\em {The Truncated Matrix-Valued
  K-Moment Problem on Rd, Cd and Td}}, Trans. of Amer. Math. Society, 365
  (2013), pp.~5393--5430.

\bibitem{kuang-2017bis}
{\sc X.~Kuang, B.~Ghaddar, J.~Naoum-Sawaya, and L.~F. Zuluaga}, {\em
  {Alternative LP and SOCP Hierarchies for ACOPF Problems}}, IEEE TPS, 32
  (2017), pp.~2828--2836.

\bibitem{kuhlmann-2007}
{\sc S.~Kuhlmann and M.~Putinar}, {\em {Positive Polynomials on Fibre
  Products}}, C. R. Acad. Sci. Paris, 344 (2007), pp.~681--684.

\bibitem{toh-2017}
{\sc J.B. Lasserre, K.C. Toh, and S.~Yang}, {\em {A Bounded Degree SOS
  Hierarchy for Polynomial Optimization}}, EURO J. Comp. Optim., 5 (2017),
  pp.~87--117.

\bibitem{lasserre-2000}
{\sc J.~B. Lasserre}, {\em {Optimisation Globale et Th\'eorie des Moments}},
  C.\ R.\ Acad.\ Sci.\ Paris, S\'erie I, 331 (2000), pp.~929--934.

\bibitem{lasserre-2001}
\leavevmode\vrule height 2pt depth -1.6pt width 23pt, {\em {Global Optimization
  with Polynomials and the Problem of Moments}}, SIAM J. Optim., 11 (2001),
  pp.~796--817.

\bibitem{lasserre-2010}
\leavevmode\vrule height 2pt depth -1.6pt width 23pt, {\em {Moments, Positive
  Polynomials and Their Applications}}, no.~1 in Imperial College Press
  Optimization Series, Imperial College Press, 2010.

\bibitem{lasserre2007}
{\sc J.~B. Lasserre, Monique Laurent, and P.~Rostalski}, {\em {Semidefinite
  Characterization and Computation of Zero-Dimensional Real Radical Ideals}},
  Found. Comp. Math.,  (2007).

\bibitem{laurent2005}
{\sc M.~Laurent}, {\em {Revisiting Two Theorems of Curto and Fialkow on Moment
  Matrices}}, Proc. Amer. Math. Soc., 10 (2005), pp.~2965–--2976.

\bibitem{laurent2009}
\leavevmode\vrule height 2pt depth -1.6pt width 23pt, {\em {Sums of Squares,
  Moment Matrices and Optimization over Polynomials}}, Emerging applications of
  algebraic geometry, IMA Vol. Math. Appl., 149, Springer, New York, 10 (2009),
  pp.~157–--270.

\bibitem{mourrain2009}
{\sc M.~Laurent and B.~Mourrain}, {\em {A Generalized Flat Extension Theorem
  for Moment Matrices}}, Arch. Math. (Basel), 6 (2009), pp.~87–--98.

\bibitem{lavaei-low-2012}
{\sc J.~Lavaei and S.H. Low}, {\em {Zero Duality Gap in Optimal Power Flow
  Problem}}, IEEE TPS, 27 (2012), pp.~92--107.

\bibitem{borden-demarco-lesieutre-molzahn-2011}
{\sc B.C. Lesieutre, D.K. Molzahn, A.R. Borden, and C.L. DeMarco}, {\em
  {Examining the Limits of the Application of Semidefinite Programming to Power
  Flow Problems}}, in 49th Annu. Allerton Conf. Commun., Control, Comput.,
  2011, pp.~28--30.

\bibitem{yalmip}
{\sc J.~L\"{o}fberg}, {\em {YALMIP: A Toolbox for Modeling and Optimization in
  MATLAB}}, in {IEEE Int. Symp. Comput. Aided Contr. Syst. Des.}, 2004,
  pp.~284--289.

\bibitem{low_tutorial}
{\sc S.H. Low}, {\em {Convex Relaxation of Optimal Power Flow: Parts I \& II}},
  {IEEE Trans. Control Network Syst.}, 1 (2014), pp.~15--27.

\bibitem{luo-2010}
{\sc Z.~Luo, W.-K. Ma, A.M.-C. So, Y.~Ye, and S.~Zhang}, {\em {Semidefinite
  Relaxation of Quadratic Optimization Problems}}, IEEE Signal Process. Mag.,
  27 (2010), pp.~20--–34.

\bibitem{maricic-2003}
{\sc B.~Maricic, Z.-Q. Luo, and T.N. Davidson}, {\em {Blind Constant Modulus
  Equalization via Convex Optimization}}, IEEE Trans. Signal Process., 51
  (2003), pp.~805--–818.

\bibitem{mccullough1989}
{\sc S.~McCullough and V.~Paulsen}, {\em {A Note on Joint Hyponormality}},
  Proc. Amer. Math. Soc., 107 (1989), pp.~187–--195.

\bibitem{pscc2014}
{\sc D.K. Molzahn and I.A. Hiskens}, {\em {Moment-Based Relaxation of the
  Optimal Power Flow Problem}}, 18th Power Syst. Comput. Conf. (PSCC),  (2014).

\bibitem{dan2015}
\leavevmode\vrule height 2pt depth -1.6pt width 23pt, {\em {Mixed SDP/SOCP
  Moment Relaxations of the Optimal Power Flow Problem}}, in IEEE Eindhoven
  PowerTech, 29 June--2 July 2015.

\bibitem{mh_sparse_msdp}
\leavevmode\vrule height 2pt depth -1.6pt width 23pt, {\em {Sparsity-Exploiting
  Moment-Based Relaxations of the Optimal Power Flow Problem}}, IEEE Trans.
  Power Syst., 30 (2015), pp.~3168--3180.

\bibitem{nie-2014}
{\sc J.~Nie}, {\em {Optimality Conditions and Finite Convergence of Lasserre's
  Hierarchy}}, Math. Program., 146 (2014), pp.~97--121.

\bibitem{parrilo-2000b}
{\sc P.A. Parrilo}, {\em {Structured Semidefinite Programs and Semialgebraic
  Geometry Methods in Robustness and Optimization}}, PhD thesis, Cal. Inst. of
  Tech., May 2000.

\bibitem{parrilo-2003}
\leavevmode\vrule height 2pt depth -1.6pt width 23pt, {\em {Semidefinite
  Programming Relaxations for Semialgebraic Problems}}, Math. Program., 96
  (2003), pp.~293--320.

\bibitem{povray2013}
{\sc POV-Ray}, {\em Persistence of vision raytracer 3.7.0}.
\newblock \url{http://www.povray.org/}, Nov. 2013.

\bibitem{putinar1992}
{\sc M.~Putinar}, {\em {Sur la Complexification du Probl\`{e}me des Moments}},
  C. R. Acad. Sci. Paris Sér. I Math., 10 (1992), pp.~743–--745.

\bibitem{putinar-1993}
\leavevmode\vrule height 2pt depth -1.6pt width 23pt, {\em {Positive
  Polynomials on Compact Semi-Algebraic Sets}}, Indiana Univ. Math. J., 42
  (1993), pp.~969--984.

\bibitem{putinar-2006}
\leavevmode\vrule height 2pt depth -1.6pt width 23pt, {\em {On Hermitian
  Polynomial Optimization}}, Arch. Math., 87 (2006), pp.~41--51.

\bibitem{putinar-scheiderer-2012}
{\sc M.~Putinar and C.~Scheiderer}, {\em {Hermitian Algebra on the Ellipse}},
  Illinois J. Math., 56 (2012), pp.~213--220.

\bibitem{putinar-2013}
\leavevmode\vrule height 2pt depth -1.6pt width 23pt, {\em {Quillen Property of
  Real Algebraic Varieties}}, Muenster J. Math., 7 (2014), pp.~671--696.

\bibitem{putinar2008}
{\sc M.~Putinar and K.~Schm\"{u}dgen}, {\em {Multivariate Determinateness }},
  Indiana Univ. Math. J., 57 (2008), pp.~2931--2968.

\bibitem{quillen-1968}
{\sc D.G. Quillen}, {\em {On the Representation of Hermitian Forms as Sums of
  Squares}}, Invent. Math., 5 (1968), pp.~237--242.

\bibitem{riener-2013}
{\sc C.~Riener, T.~Theobald, L.~J. Andr\'en, and J.~B. Lasserre}, {\em
  {Exploiting Symmetries in SDP-Relaxations for Polynomial Optimization}},
  Math. of Operations Research, 38 (2013), pp.~122--141.

\bibitem{rudin-1987}
{\sc W.~Rudin}, {\em {Real and Complex Analysis}}, Math. Ser., Third Edition,
  McGraw Hill Int. Ed., 1987.

\bibitem{schmudgen-1991}
{\sc K.~Schm\"udgen}, {\em {The K-Moment Problem for Semi-Algebraic Sets}},
  Math. Ann., 289 (1991), pp.~203--206.

\bibitem{singer-2011}
{\sc A.~Singer}, {\em {Angular Synchronization by Eigenvectors and Semidefinite
  Programming}}, Appl. Comput. Harmon. Anal., 30 (2011), pp.~20--–36.

\bibitem{stochel1998}
{\sc J.~Stochel and F.~H. Szafraniec}, {\em {The Complex Moment Problem and
  Subnormality: a Polar Decomposition Approach}}, J. Funct. Anal., 159 (1998),
  pp.~432–--491.

\bibitem{taylor2015}
{\sc J.A. Taylor}, {\em {Convex Optimization of Power Systems}}, Cambridge
  University Press, 2015.

\bibitem{toh2012}
{\sc K.-C. Toh, M.~J. Todd, and R.~H. T{\"u}t{\"u}nc{\"u}}, {\em On the
  Implementation and Usage of SDPT3 -- A Matlab Software Package for
  Semidefinite-Quadratic-Linear Programming, Version 4.0}, Springer US, Boston,
  MA, 2012, pp.~715--754.

\bibitem{toker-1998}
{\sc O.~Toker and H.~Ozbay}, {\em {On the Complexity of Purely Complex Mu
  Computation and Related Problems in Multidimensional Systems}}, IEEE Trans.
  Automat. Control, 43 (1998), pp.~409--414.

\bibitem{vasilescu-2009}
{\sc F.~H. Vascilescu}, {\em {Subnormality and Moment Problems}}, Extracta
  mathematicae, 24 (2009), pp.~167--–186.

\bibitem{waki-2006}
{\sc H.~Waki, S.~Kim, M.~Kojima, and M.~Muramatsu}, {\em {Sums of Squares and
  Semidefinite Program Relaxations for Polynomial Optimization Problems with
  Structured Sparsity}}, SIAM J. Optim., 17 (2006), pp.~218--242.

\bibitem{weisser-2017}
{\sc T.~Weisser, J.B. Lasserre, and K.C. Toh}, {\em {Sparse-BSOS: a Bounded
  Degree SOS Hierarchy for Large Scale Polynomial Optimization with Sparsity}},
  Mathematical Progamming Computation, 5 (2017), pp.~1--32.

\bibitem{zheng2017}
{\sc A.~Papachristodoulou Y.~Zheng, G.~Fantuzzi}, {\em {Fast ADMM for
  Sum-of-Squares Programs Using Partial Orthogonality}},
  \url{https://arxiv.org/pdf/1708.04174.pdf},  (2017).

\bibitem{zag2015}
{\sc S.~Zagorodnyuk}, {\em {On the Truncated Operator Trigonometric Moment
  Problem}}, Concr. Oper., 2 (2015), pp.~37--46.

\bibitem{murillosanchez-thomas-zimmerman-2011}
{\sc R.~Zimmerman, C.~Murillo-S\'anchez, and R.~Thomas}, {\em {MATPOWER:
  Steady-State Operations, Planning, and Analysis Tools for Power Systems
  Research and Education}}, IEEE TPS, 99 (2011), pp.~1--8.

\end{thebibliography}
\bibliographystyle{siam}

\end{document}